\documentclass[11pt,letterpaper]{amsart}

\usepackage{amssymb}
\usepackage{amsmath}

\usepackage{graphicx}

\makeatletter
\@namedef{subjclassname@2020}{%
  \textup{2020} Mathematics Subject Classification}
\makeatother

\newtheorem{theorem}{Theorem}[section]
\newtheorem{corollary}[theorem]{Corollary}
\newtheorem{lemma}[theorem]{Lemma}
\newtheorem{proposition}[theorem]{Proposition}

\theoremstyle{definition}
\newtheorem{definition}[theorem]{Definition}

\numberwithin{equation}{section}

\begin{document}




\title[$H^p$ spaces of separately $(\alpha, \beta)-$harmonic functions in the polydisc]{$H^p$ spaces of separately $(\alpha, \beta)-$harmonic functions in the unit polydisc}

\author[J. Gaji\'c]{Jelena Gaji\'c}
\address{Faculty of Natural Sciences and Mathematics \\ University of Banja Luka\\ Mladena Stojanovi\'ca 2\\ 78000 Banja Luka, Bosnia and Herzegovina}
\email{jelena.gajic.mm@gmail.com}

\author[M. Mateljevi\'c]{Miodrag Mateljevi\'c}
\address{Department of Mathematics\\ University of Belgrade\\
Studentski Trg 16\\
11000 Belgrade, Serbia}
\email{miodrag@matf.bg.ac.rs}

\author[M. Arsenovi\'c]{Milo\v s Arsenovi\'c}
\address{Department of Mathematics\\ University of Belgrade\\
Studentski Trg 16\\
11000 Belgrade, Serbia}

\email{arsenovic@matf.bg.ac.rs}


\


\date{}

\begin{abstract}

We prove existence and uniqueness of a solution of the Dirichlet problem  for separately $(\alpha, \beta)$ - harmonic functions on the unit polydisc $\mathbb D^n$ with  boundary data in $C(\mathbb T^n)$ using 
$(\alpha, \beta)$ - Poisson kernel 
$P_{\alpha, \beta} (z, \zeta)$. 
A characterization by  hypergeometric functions of separately
$(\alpha, \beta)$ - harmonic functions which are also $m$ - homogeneous is given, this characterization is used to obtain series expansion of separately $(\alpha, \beta)$ - harmonic functions. Basic $H^p$ theory of such functions is developed: integral representations by measures and $L^p$ functions on $\mathbb T^n$,  
norm and weak$^\ast$ convergence at the distinguished boundary $\mathbb T^n$.
Weak $(1,1)$ - type estimate for a restricted non-tangential maximal function $M_{A, B}^{NT}$ is derived.
Slice functions $u(z_1, \ldots, z_k, \zeta_{k+1}, \ldots, \zeta_n)$, where some of the variables are fixed, are shown to belong in
the appropriate space of separately $(\alpha', \beta')$ - harmonic functions of $k$ variables.
We  prove a Fatou type theorem on a. e. existence of restricted non-tangential limits for these functions and a corresponding result for unrestricted limit at a point in $\mathbb T^n$. Our results extend earlier results for $(\alpha, \beta)$ - harmonic functions in the disc and for $n$ - harmonic functions in $\mathbb D^n$. 

\end{abstract}



\maketitle

\section{Introduction and notation}\label{Intro}

The first line of research that motivates our investigations is concerned with $(\alpha, \beta)$ - harmonic functions, i.e. functions
$u$ satisfying equation $L_{\alpha, \beta} u = 0$ where 
$L_{\alpha, \beta} = \Delta + \alpha L_1 + \beta L_2 + Q_{\alpha, \beta}$. Here $\Delta$ is the Laplace - Beltrami operator on a domain $\Omega$, $L_1$ and $L_2$ are linear differential operators of first order and $Q_{\alpha, \beta}$ is a zero order operator. An early paper touching on this topic is \cite{Ge} where $\Omega$ is the Siegel upper half plane ${\bf U^{n+1}}$. 
There $(\alpha, \beta)$ - harmonic functions were used as a tool for obtaining results on $H^p$ theory on the Heisenberg group
$\partial  {\bf U^{n+1}} = {\bf H^n}$.
The case where $\Omega$ is the unit ball in $\mathbb C^n$ was considered in \cite{ABC} where series expansion of 
$(\alpha, \beta)$ - harmonic functions was derived and $H^p$ theory results for such functions were obtained. A detailed 
investigation of series expansion of $(\alpha, \beta)$ - harmonic functions in the unit disc can be found in \cite{KO}. Let us 
mention that special one parameter cases also appeared, see for example \cite{OW}, \cite{MA} and \cite{OP}. 
In the cited papers the operators $L_{\alpha, \beta}$ are uniformly elliptic, but degenerate near the boundary of the domain: the lower order terms $\alpha L_1, \beta L_2$ and  $Q_{\alpha, \beta}$ are not dominated by the principal part unless 
$\alpha = \beta = 0$.


The second line of research relevant for us is the study of $n$ - harmonic functions, that is functions defined on the unit polydisc $\mathbb D^n \subset \mathbb C^n$ which are harmonic in each of the $n$ variables $z_1, \ldots, z_n$. Such functions 
appeared already in 1939 (for $n = 2$), see \cite{MZ}. More details on $n$ - harmonic functions can be found in Chapter XVII of \cite{Z} and in \cite{Rud2}. A probabilistic approach to some aspects of $H^p$ theory of $n$ - harmonic functions is given in \cite{GuSt}, more recent paper on $n$ - harmonic functions is \cite{Seur}.

In this paper we merge these two avenues of investigation. We deal with functions $u$ in the unit polydisc which satisfy a system of PDEs: 
$L_{\alpha_j, \beta_j} u = 0$, $1 \leq j \leq n$. Each of the operators $L_{\alpha_j, \beta_j}$ acts on the $z_j = x_j + i y_j$
variable and annihilates the same functions as the Laplacian if $\alpha_j = \beta_j = 0$. Our results extend one dimensional results from \cite{OW}, \cite{KO}  and \cite{OP} on a single operator $L_{\alpha, \beta}$ to
the multidimensional case of a system of operators, these results are also generalizations of results on $n$ - harmonic functions.

The paper is organized as follows. In the present section we establish notation used throughout the paper. In Section \ref{SepSec} we define separately $(\alpha, \beta)$ - harmonic functions in the unit polydisc and review known results in one dimensional case that we need. The rest of Section \ref{SepSec} is devoted to series expansion of separately $(\alpha, \beta)$ - harmonic functions. In Section \ref{HpSec} we obtain
existence and uniqueness result for the Dirichlet problem with continuous boundary data. Next $H^p$ spaces
of separately $(\alpha, \beta)$ - harmonic functions are defined, integral representations are obtained and convergence in 
norm and in the weak$^\ast$ sense is discussed. The final Section \ref{FatouSec} contains results on relevant maximal functions, these are used to derive Fatou type theorems. Our results are on the so called restricted limits, we plan to treat unrestricted linits
in another publication. For a discussion of restricted and non-restricted limits, as well as approach regions, in a more general setting of symmetric spaces, see \cite{St1}.

We use standard notation and terminology. The boundary of the unit disc 
$\mathbb D = \{ z \in \mathbb C : \vert z \vert < 1 \}$ is $\mathbb T = \{ \zeta \in \mathbb C : \vert \zeta \vert = 1 \}$.  
Given $\zeta = e^{i\varphi}$ in $\mathbb T$ we denote the Stolz region at $\zeta$ of aperture $0 < A < +\infty$ by
$$
S_A (e^{i\varphi}) =  \{  r e^{i (\varphi - \theta)} : \vert \theta \vert    \leq A (1-r)    \}.
$$
Non-tangential limit at a boundary point is denoted by $NT - \lim$. 

The unit polydisc $\mathbb D^n$ in $\mathbb C^n$ is the 
Cartesian product of $n$ copies of $\mathbb D$,
and the unit torus $\mathbb T^n$ in $\mathbb C^n$ is the Cartesian product of $n$ copies of $\mathbb T$.
If $n>1$, $\mathbb T^n$ is only a small part of the boundary of $\mathbb D^n$ and  the unit torus is usually called the distinguished boundary  of $\mathbb D^n$. We set ${\bf 1} = (1, \ldots, 1)$.

For $z =  (z_1, \ldots, z_n) \in \mathbb C^n$ we set $ \overline z = (\overline{z_1}, \ldots, \overline{z_n})$. For
$z, w \in \mathbb C^n$ we set $z \cdot w = (z_1 w_1, \ldots, z_n w_n)$.  We also write $z = (z', z_n)$ where 
$z = (z_1, \ldots, z_n)$ is in $\mathbb C^n$ and $z' = (z_1, \ldots, z_{n-1})$ is in $\mathbb C^{n-1}$. Euclidean balls in 
$\mathbb C^n$ are denoted by $B(z, r)$.

We will use the normalized Haar measure $ m_n$ on $\mathbb T^n$, therefore we have $m_n(\mathbb T^n)=1$. The
space of complex Borel measures on $\mathbb T^n$ is denoted by $\mathcal M (\mathbb T^n)$, it is the dual space of 
$C(\mathbb T^n)$. The pairing of duality is given by
$$ \langle \varphi, d\mu \rangle = \int_{\mathbb T^n} \varphi d\mu, \qquad \varphi \in C(\mathbb T^n), \quad
\mu \in \mathcal M (\mathbb T^n).
$$
The total variation of $\mu$ is denoted by $\| \mu \|$, it is equal to $\vert \mu \vert (\mathbb T^n)$.

For $z$ in $\mathbb D^n$ and a function $u : \mathbb D^n \to \mathbb C$
we define $u_z : \mathbb T^n \to \mathbb C$ by $u_z (\zeta) = u(z \cdot \zeta)$.
If $z = (r, r, \ldots, r)$ where $0 \leq r < 1$ we write $u_r$ instead of $u_z$. In particular, if ${\bf r} = (r_1, \ldots, r_n) \in
[0, 1)^n$, then $u_{\bf r}(\zeta) = u(r_1 \zeta_1, \ldots, r_n\zeta_n)$. Also, for $\zeta \in \mathbb T^n$ we define
a function $\zeta {\bf \cdot} u$ on $\mathbb D^n$ by $\zeta {\bf \cdot} u(z) = u(\zeta \cdot z)$.

We use notation $x^+ = \max (x, 0)$ and $x^- = \max (-x, 0)$ for a real number $x$. For $m = (m_1, \ldots, m_n)$ in
$\mathbb Z^n$ we set $m^+ = (m_1^+, \ldots, m_n^+) \in \mathbb Z^n_+$ and similarly
 $m^- = (m_1^-, \ldots, m_n^-) \in \mathbb Z^n_+$. $\vert m \vert = \vert m_1 \vert + \cdots + \vert m_n \vert$
is the order of a multiindex $m \in \mathbb Z^n$. Also, $m! = m_1! \cdot \ldots \cdot m_n!$ for $m \in \mathbb Z^n_+$.
For $p = (p_1, \ldots, p_n)$ in $\mathbb Z_+^n$ and $z \in \mathbb C^n$ we set 
$z^p = z_1^{p_1} \cdot \ldots \cdot z_n^{p_n}$,  $\overline z^p = \overline z_1^{p_1} \cdot \ldots \cdot \overline z_n^{p_n}$.

The order of a multiindex $(p, q)$ of length $2n$ is $\vert (p, q) \vert = \vert p \vert + \vert q \vert$. We set
$$
\partial^{(p, q)} = \frac{\partial^{p_1}}{\partial z_1^{p_1}} \cdots \frac{\partial^{p_n}}{\partial z_n^{p_n}}
 \frac{\partial^{q_1}}{\partial \overline z_1^{q_1}} \cdots \frac{\partial^{q_n}}{\partial \overline z_n^{q_n}}, \quad
\partial^p = \partial^{(p, 0)}, \quad {\overline\partial}^q = \partial^{(0, q)}.
$$
Here $\partial/\partial z_j = 2^{-1}(\partial/\partial x_j - i \partial/\partial y_j)$ and $\partial/\partial \overline z_j = 2^{-1}
(\partial x_j + i \partial y_j)$ are standard first order linear differential operators.

The space $C^\infty (\mathbb D^n)$ is a locally convex complete metrizable topological vector space, it is endowed with the
following family of seminorms:
$$
\| u \|_{p, q, K} = \max_{z \in K} \vert \partial^{(p, q)} u(z) \vert, \qquad K \subset \mathbb D^n \quad \mbox{is compact},
\quad p, q \in \mathbb Z^n_+.
$$

For complex numbers $ a, b$ and $c$ such that $c \not= 0, -1, -2, \ldots$ hypergeometric function
$F(a, b; c; z)$  is defined by
\begin{equation}\label{hypergeom}
F(a, b; c; z)  = \sum_{k=0}^\infty \frac{(a)_k (b)_k}{(c)_k} \frac{z^k}{k!}, \qquad \vert z \vert < 1
\end{equation}
where $(a)_k $ is the Pochhammer symbol: 
$$
(a)_0 = 1 \qquad \mbox{and} \qquad  (a)_k = a (a+1) \cdots (a+k-1)\quad \mbox{\rm for} \quad k = 1, 2, \ldots.
$$     
If $\Re (a + b - c) < 0$ then the series in \eqref{hypergeom} converges absolutely for $\vert z \vert \leq 1$.
More details about the hypergeometric functions can be found in \cite{B}.

\section{Separately $(\alpha, \beta)$ - harmonic functions in $\mathbb D^n$}\label{SepSec}

\begin{definition}\label{defalphabeta}
Let $ \alpha=(\alpha_1, \ldots, \alpha_n)\in \mathbb{C}^n, \beta=(\beta_1, \ldots, \beta_n)\in \mathbb{C}^n$.  
A function $ u $ in $C^2( \mathbb D^n)$ is separately $ (\alpha, \beta)$ - harmonic if $L_{\alpha_j, \beta_j} u=0$ for
$1\leq j\leq n$, where 
$$ L_{\alpha_j, \beta_j} =(1-\vert z_j\vert^2) \frac{\partial^2 }{\partial z_j \partial \overline{z_j}}
+ \alpha_j z_j \frac{\partial }{\partial z_j}
+ \beta_j \overline{z_j}\frac{\partial }{\partial \overline{z_j}}-\alpha_j \beta_j, \qquad 1 \leq j \leq n.$$  
We denote the vector space of all separately $ (\alpha, \beta)$ - harmonic functions in $ \mathbb{D}^n $ by
$sh_{\alpha,\beta}(\mathbb{D}^n )$.
\end{definition}

Note that $L_{\alpha_j, \beta_j}$ is a second order linear partial differential operator which acts only with respect to 
$z_j = x_j + iy_j$ variable, hence the term "separately". The first order coefficients $\alpha_j z_j$ and $\beta_j \overline{z_j}$
are not dominated by $1 - \vert z_j \vert^2$, unless $\alpha_j = \beta_j = 0$, thus although the operator is uniformly elliptic it
is degenerate. If $\alpha_j = \beta_j = 0$, then $L_{\alpha_j, \beta_j} = 4^{-1} (1 - \vert z_j \vert^2) \Delta_j$, where 
$\Delta_j$ is the Laplace operator with respect to $x_j$ and $y_j$ variables. Thus $sh_{0, 0}(\mathbb D^n)$ is the space of
$n$ - harmonic functions. Even in this special case we do not get an algebra: a square of a harmonic function is not a harmonic
function in general.

\begin{lemma}\label{cinfty}
$sh_{\alpha, \beta} (\mathbb D^n) \subset C^\infty (\mathbb D^n)$.
\end{lemma}
\begin{proof}
Any separately $(\alpha, \beta)$ - harmonic function $u$ is a solution of a linear elliptic
partial differential equation
$Lu = 0$, where $L = L_{\alpha_1, \beta_1} + \cdots + L_{\alpha_n, \beta_n}$ has $C^\infty$ coefficients. Therefore, by elliptic regularity, $u$ is in $C^\infty (\mathbb D^n)$.
\end{proof}


\begin{theorem} [\cite{KO}, Theorem 1.4 and Theorem 6.4.]  \label{Uab}
Let $\alpha, \beta \in \mathbb C$. Then the function
\begin{equation}\label{ualphabeta}
 u_{\alpha, \beta}(z)= 
 \frac{(1-\vert z\vert^2)^{\alpha+\beta+1}}{(1-z)^{\alpha+1}
(1-\overline{z} )^{\beta+1}}, \qquad \vert z \vert < 1
\end{equation} 
is $(\alpha, \beta)$ - harmonic in $\mathbb D$. Moreover, if $\Re \alpha + \Re \beta>-1$, the following estimate holds:
\begin{equation}\label{L1est}
\frac{1}{2\pi} \int_{-\pi}^\pi \vert u_{\alpha, \beta}(r e^{i\theta})\vert d\theta \leq e^{\frac{\pi}{2}
\vert \Im  \alpha-\Im \beta\vert} 
\frac{\Gamma(\Re \alpha + \Re \beta +1)}{\Gamma^2\left(\frac{\Re \alpha + \Re \beta}{2} +1\right)}, \qquad 0 \leq r < 1.
\end{equation}
\end{theorem}

We also define, for complex $\alpha, \beta \not\in \mathbb Z^-  = \{-1, -2, \ldots \}$ and $\Re \alpha + \Re \beta > -1$:
\begin{equation}\label{coef}
 c_{\alpha, \beta}=\frac{\Gamma(\alpha+1)\Gamma(\beta+1)}{\Gamma(\alpha + \beta + 1)},
 \end{equation}
\begin{equation}\label{valphabeta}
v_{\alpha, \beta}(z) = c_{\alpha, \beta} u_{\alpha, \beta}(z), \qquad z \in \mathbb D.
\end{equation}

Since, for $\zeta \in \mathbb C$, $\partial_z[u(\zeta z)] = \zeta (\partial_z u) (\zeta z)$ and $\partial_{\overline z} [u(\zeta z)] =
\overline\zeta (\partial_{\overline z} u)(\zeta z)$ we get
\begin{equation*}
L_{\alpha, \beta}[u(\zeta z)] = 
(1 - \vert z \vert^2) \vert \zeta \vert^2  \frac{\partial^2 u}{\partial z \partial \overline z} (\zeta z)
+ \alpha \zeta z (\partial_z u)(\zeta z) +
\beta \overline {\zeta z} (\partial_{\overline z} u)(\zeta z) - \alpha\beta u(\zeta z).
\end{equation*}
A consequence of this formula is the following rotational invariance of $L_{\alpha, \beta}$, which appeared in \cite{KO}: 
$L_{\alpha, \beta}(\zeta \cdot u) = \zeta \cdot (L_{\alpha, \beta} u)$ if $\vert \zeta \vert = 1$. It follows that
\begin{equation}\label{zetamod1}
L_{\alpha_j, \beta_j}(\zeta \cdot u) = \zeta \cdot (L_{\alpha_j, \beta_j} u), \qquad \zeta \in \mathbb T^n, \quad 1 \leq j \leq n.
\end{equation}
Hence,  for $\vert \zeta \vert = 1$, the function $f(z) = u_{\alpha, \beta}(\zeta z)$ is $(\alpha, \beta)$  -  harmonic in 
$\mathbb D$.  On the other hand, taking $\zeta = r$ we obtain
\begin{equation}\label{zetar}
L_{\alpha, \beta}[u(rz)] =  (L_{\alpha, \beta} u)(rz) - (1-r^2) \frac{\partial^2 u}{\partial z \partial \overline z} (rz), \qquad 
0 < r < 1,
\end{equation}
which shows that for $(\alpha, \beta)$ - harmonic function $u$ the function 
$z \mapsto u(rz)$ need not be $(\alpha, \beta)$ - harmonic.

%
From this point on we assume that $\alpha, \beta \in \mathbb C^n$ satisfy conditions 
\begin{equation}\label{conditions}
\Re \alpha_j + \Re \beta_j  > -1 \quad  \mbox{\rm and} \quad \alpha_j, \beta_j \not\in \mathbb Z^- \qquad 
\mbox{\rm for all} \qquad 1 \leq j \leq n.
\end{equation}

\begin{definition}\label{alphabetakernel}
We define $(\alpha, \beta)$ - Poisson kernel by 
\begin{equation}\label{abkernel}
 P_{\alpha,\beta} (z,\zeta) = \prod_{j=1}^n  c_{\alpha_j, \beta_j}
\frac{(1-\vert z_j\vert^2)^{\alpha_j+\beta_j+1}}{(1-z_j \overline{\zeta_j})^{\alpha_j+1}
(1-\overline{z_j} \zeta_j)^{\beta_j+1}}, \quad z \in \mathbb D^n, 
\zeta \in \mathbb T^n,
\end{equation}
where the constants $ c_{\alpha_j, \beta_j}$ are defined in \eqref{coef}.
\end{definition}
If $ \Re\alpha_j=\alpha_j=\beta_j>-\frac{1}{2} $    for all $ j=1,\ldots, n, $ then $ P_{\alpha, \beta}$ is real and positive. 
This case will be important in Section \ref{FatouSec}. On the other hand, the case $n=1$, $\alpha = 0$ and $\beta > -1$ appeared in \cite{OW}. Setting
\begin{equation}\label{functionV}
 V_{\alpha, \beta}(z) = \prod_{j=1}^n v_{\alpha_j, \beta_j}(z_j), \qquad z \in \mathbb  D^n
\end{equation}
we can write, for $z \in \mathbb D^n$ and  $\zeta \in \mathbb T^n$:
\begin{equation}\label{prodpoisson}
P_{\alpha, \beta}(z, \zeta) = \prod_{j=1}^n P_{\alpha_j, \beta_j}(z_j, \zeta_j) = \prod_{j=1}^n 
v_{\alpha_j, \beta_j}(z_j \overline \zeta_j) = V_{\alpha, \beta}(z \cdot \overline\zeta).
\end{equation}
By Theorem  \ref{Uab} and \eqref{zetamod1} we see that  $u_{\alpha_j, \beta_j}(\overline \zeta_j z_j)$ is 
$(\alpha_j, \beta_j)$ - harmonic  for $1 \leq j \leq n$. Thus the following lemma immediately follows.
\begin{lemma}\label{Psep}
The $(\alpha, \beta)$ - Poisson kernel $P_{\alpha, \beta}(z, \zeta)$ is separately $(\alpha, \beta)$ - harmonic 
with respect to $z \in \mathbb D^n$ for each fixed $\zeta \in \mathbb T^n$.
\end{lemma}

\begin{definition}\label{Poissonext}
We define $(\alpha, \beta)$ - Poisson integral of $\mu \in \mathcal M(\mathbb T^n)$ by the following formula:
\begin{equation}\label{poisson}
P_{\alpha, \beta}[d\mu](z) = \int_{\mathbb T^n} P_{\alpha, \beta}(z, \zeta) d\mu(\zeta), \qquad z \in \mathbb D^n.
\end{equation}
For $f \in L^1(\mathbb T^n)$ we define $P_{\alpha, \beta}[f] = P_{\alpha, \beta}[fdm_n]$, that is
\begin{equation}\label{poissonL1}
P_{\alpha, \beta}[f](z) = \int_{\mathbb T^n} P_{\alpha, \beta}(z, \zeta) f(\zeta) dm_n(\zeta), \qquad z \in \mathbb D^n.
\end{equation}
\end{definition}
We use the same notation $P_{\alpha, \beta}$ for the operator and its integral kernel, this abuse of notation is quite harmless.

Let $\mu \in \mathcal M(\mathbb T^n)$ and ${\bf r} = (r_1, \ldots, r_n) \in [0, 1)^n$. For $u = P_{\alpha, \beta}[d\mu]$
we have
\begin{equation}\label{convPmu1}
u(r_1 e^{i t_1}, \ldots, r_n e^{i t_n}) = \int_{\mathbb T^n} \prod_{j=1}^n v_{\alpha_j, \beta_j} (r_j e^{i (t_j - \zeta_j)}) d\mu (\zeta)
\end{equation}
and therefore
\begin{equation}\label{convPmu2}
(P_{\alpha, \beta}[d\mu])_{\bf r} = (V_{\alpha, \beta})_{\bf r} \ast d\mu
\end{equation}
where convolution is taken with respect to the group structure of $\mathbb T^n$. Taking $d\mu = f dm_n$ in 
\eqref{convPmu2} we obtain
\begin{equation}\label{convPf}
(P_{\alpha, \beta}[f])_{\bf r} = (V_{\alpha, \beta})_{\bf r} \ast f, \qquad f \in L^1(\mathbb T^n).
\end{equation}

From \eqref{ualphabeta}, \eqref{L1est},
\eqref{coef}, \eqref{valphabeta} and \eqref{functionV} using  Fubini's theorem we obtain
\begin{equation}\label{L1V}
\|  (V_{\alpha, \beta})_{\bf r} \|_{L^1(\mathbb T^n)} \leq K(\alpha, \beta), \qquad {\bf r} \in [0, 1)^n
\end{equation}
where 
\begin{equation}\label{Kalphabeta}
K(\alpha, \beta) = e^{ \frac{\pi}{2} \sum_{j=1}^n \mid \Im \alpha_j - \Im \beta_j \mid } \prod_{j=1}^n \left\vert
 \frac{\Gamma(\alpha_j + 1)\Gamma(\beta_j + 1)}{\Gamma(\alpha_j + \beta_j + 1)} \right\vert
\frac{\Gamma(\Re \alpha_j + \Re \beta_j +1)}{\Gamma^2 \left(\frac{\Re \alpha_j + \Re \beta_j}{2} +1\right)}.
\nonumber
\end{equation}
The inequality \eqref{L1V} can also be written in the form
\begin{equation}\label{L1p}
\int_{\mathbb{T}^n}\vert P_{\alpha,\beta}(z, \zeta) \vert dm_n(\zeta) \leq K(\alpha, \beta), \qquad z \in \mathbb D^n.
\end{equation}

An immediate consequence of Lemma \ref{Psep} is the next proposition.
\begin{proposition}\label{shPofmeasures}
For any complex Borel measure $\mu$ on $\mathbb T^n$ its $(\alpha, \beta)$ - Poisson integral 
$u = P_{\alpha, \beta} [d\mu]$ is separately $(\alpha, \beta)$ - harmonic function on $\mathbb D^n$.
\end{proposition} 

Let $ \varphi \in C(\mathbb T)$. Consider the following Dirichlet problem: Find a  function $u$ in 
$C(\overline{\mathbb D})$ such that $L_{\alpha, \beta} u = 0$ in $\mathbb D$ and $u = \varphi$ on $\mathbb T$.

\begin{theorem}[\cite{KO}, Theorem 7.1]\label{KODP}
The above Dirichlet problem has a unique solution $u \in C(\overline{\mathbb D})$, it is given by the following formula:
 \begin{equation}\label{Dirpr}
u(z) = 
\begin{cases}
P_{\alpha, \beta}[\varphi](z), & z \in \mathbb D \\
\varphi (z), & z \in \mathbb T.
\end{cases}
\end{equation}
 \end{theorem}
We generalize this result to separately $(\alpha, \beta)$ - harmonic functions with continuous boundary data on $\mathbb T^n$,
see Proposition \ref{Dirchndim} below.

\begin{theorem}[\cite{KO}, Theorems 5.1 and 5.3] \label{expansion}
Let $u$ be $(\alpha, \beta)$ - harmonic in $\mathbb D$. Then
\begin{align}\label{exp}
u(z)  = & \sum_{k=0}^\infty \frac{\partial^k u}{k!}(0) F(-\alpha, k - \beta; k+1; \vert z \vert^2) z^k \nonumber
\\
& +  \sum_{k=1}^\infty \frac{\overline\partial^k u}{k!}(0) F(k - \alpha, -\beta ; k+1; \vert z \vert^2) \overline z^k,
\qquad \vert z \vert < 1.
\end{align}
Convergence of the above series is in the topology of $C^\infty (\mathbb D)$.
\end{theorem}

Our first task is to extend Theorem \ref{expansion} to separately $(\alpha, \beta)$ - harmonic functions.

We say $(p, q) \in \mathbb Z^n_+ \times \mathbb Z^n_+$ is pure if 
$p \cdot q = (p_1 q_1, \ldots, p_n q_n) = \bf 0$. We denote the set of all pure multiindexes $(p, q)$ by $H$. We will use
without comment the natural bijection $\mathbb Z^n \ni m \mapsto (m^+, m^-) \in H$ to change the index of summation.
Let $z^{(m)}$ denote $z^{m^+} {\overline z}^{m^-}$ where $m \in \mathbb Z^n$ and $z \in \mathbb C^n$.

Let us assume, for a moment, that $n=1$. Then for a pure $(p, q)$, which in this case means $p=0$ or $q=0$, we set
$$
F_{p,q} (z) = F(p-\beta, q - \alpha; p + q + 1; z), \qquad \vert z \vert < 1.
$$

With this notation at hand we can write \eqref{exp} in the following form
\begin{equation}\label{exp1}
u(z) = \sum_{(p,q) \in H} \frac{ \partial^{(p, q)}u}{p!q!} (0) F_{p,q}(\vert z \vert^2) z^p \overline z^q, 
\qquad \vert z \vert < 1.
\end{equation}

Returning to the general case $n \geq 1$, for $(p, q) \in H$ we define
\begin{equation}\label{Fpq}
\mathcal F_{p, q} (z) = \prod_{j=1}^n F_{p_j, q_j} ( \vert z_j \vert^2), \qquad z \in \mathbb D^n.
\end{equation}

\begin{theorem}\label{expn}
Let $u$ be separately $(\alpha, \beta)$ - harmonic in $\mathbb D^n$. Then
\begin{equation}\label{exp1}
u(z) = \sum_{(p,q) \in H} \frac{ \partial^{(p, q)}u}{p!q!} (0) \mathcal F_{p,q} (z) z^p \overline z^q, 
\qquad z \in \mathbb D^n
\end{equation}
and convergence is in the topology of the space $C^\infty (\mathbb D^n)$.
\end{theorem}

We prove this theorem closely following methods used by 
Klintborg and Olofsson (\cite{KO}, $n=1$). We need some preliminary concepts and results.

\begin{definition}\label{Mhomog}
A continuous function $u$ on $\mathbb D^n_0 = \{ z \in \mathbb D^n : \prod_j z_j \not= 0 \}$ 
is said to be $m$ - homogeneous, where $m \in \mathbb Z^n$, if
\begin{equation}\label{mhomog}
(\zeta {\bf \cdot} u)(z)  = u(\zeta \cdot z) = \zeta^m u(z), \qquad z \in \mathbb D^n_0, \quad \zeta \in \mathbb T^n.
\end{equation} 
This means $u(e^{i\theta_1} z_1, \ldots, e^{i\theta_n} z_n) = e^{im_1 \theta_1} \cdot \ldots \cdot e^{im_n\theta_n} u(z)$ 
for all $z \in \mathbb D^n_0$ and $\theta_j \in \mathbb R$.
\end{definition}

Let us note, see \cite{KO}, that for every $(p, q) \in H$ the function $\mathcal F_{p,q}(z)  z^p \overline z^q$ is separately
$(\alpha, \beta)$ - harmonic in $\mathbb D^n$ and $(p-q)$ - homogeneous.

\begin{definition}\label{MFourier}
For  $u \in C(\mathbb D^n)$ and $m \in \mathbb Z^n$ the $m$ - th homogeneous component $u_m$ of $u$ 
is defined by
\begin{equation}\label{mForuier}
u_m(z) = \int_{\mathbb T^n} \overline\zeta^m u(\zeta \cdot z) dm_n(\zeta), \qquad z \in \mathbb D^n.
\end{equation}
\end{definition}

$u_m$ is $m$ - homogeneous by a simple change of variables. If $u \in C^k(\mathbb D^n)$ then, by differentiation under the
integral sign, $u_m \in C^k(\mathbb D^n)$ for $0 \leq k \leq \infty$.


\begin{proposition}\label{Fouriervector}
Let $u_m(z)$ be the $m$ - th homogeneous component of $u \in C^\infty (\mathbb D^n)$ for $m \in \mathbb Z^n$. Then
\begin{equation}\label{fouriervector}
\sum_{m \in \mathbb Z^n} u_m(z) = u(z) \qquad \mbox{in the topology of the space} \quad C^\infty (\mathbb D^n).
\end{equation}
\end{proposition}
\begin{proof}
We fix a compact $K = r \overline{\mathbb D}^n$ and $(\gamma, \delta) \in \mathbb Z_+^n \times \mathbb Z^n_+$. 
For a given 
$m \in \mathbb Z^n$ we set $l_j = N$ if $m_j \not= 0$ and $l_j = 0$ otherwise. This gives us $l \in \mathbb Z_+^n$ with
$\vert l \vert \leq Nn$. We can differentiate under the integral sign and perform partial integration to obtain:

\begin{align*}
\partial_z^{(\gamma, \delta)} u_m(z) & = \int_{-\pi}^\pi \cdots \int_{-\pi}^\pi 
\partial_z^{(\gamma, \delta)} u(e^{i\theta_1}z_1, \ldots, e^{i\theta_n}z_n) \prod_{j=1}^n
e^{-i m_j\theta_j} \frac{d\theta_1 \ldots d\theta_n}{ (2\pi)^n} \\
& = \frac{(2\pi)^{-n}}{(im_1)^{l_1} \cdot \ldots \cdot (im_n)^{l_n}} \int_{-\pi}^\pi \cdots \int_{-\pi}^\pi 
\partial^l_\theta  \partial_z^{(\gamma, \delta)} u  (e^{i\theta_1}z_1, \ldots, e^{i\theta_n}z_n)  \\
& \prod_{j=1}^n e^{-i m_j\theta_j} d\theta_1 \ldots d\theta_n, \qquad z \in \mathbb D^n.
\end{align*}
This implies, for $z \in K$:
\begin{equation}\label{estcomp}
\vert \partial_z^{(\gamma, \delta)} u_m(z) \vert 
\leq C  \max_{\vert (s, t) \vert \leq \vert \gamma \vert + \vert \delta \vert + Nn} 
\| \partial^{(s, t)} u \|_{C(K)} \prod_{j=1}^n \frac{1}{m_j^N}, \qquad N \in \mathbb N,
\end{equation}
in the above product $1/m_j^N$ is interpreted as $1$ if $m_j = 0$. 
Since the multiple series $\sum_{m \in \mathbb Z^n_+} m_1^{-2}\cdot \ldots \cdot m_n^{-2}$ converges, this ensures  convergence in $C^\infty (\mathbb D^n)$ topology of the multiple series in \eqref{fouriervector}. The sum converges, for a fixed $z \in \mathbb D^n$ to $u(z)$. Indeed, $u_m(z)$ is the $m$ - th Fourier coefficient of the function $u_z \in C^\infty (\mathbb T^n)$ and therefore $u_z(\zeta) = \sum_{m \in \mathbb Z^n} u_m(z) \zeta^m$. Taking $\zeta = {\bf 1}$ we obtain \eqref{fouriervector}.
\end{proof}


\begin{lemma}[\cite{KO}, Proposition 4.2 for $n=1$]\label{alphabetamhom}
Let $u \in sh_{\alpha, \beta} (\mathbb D^n)$. Then for every $m \in \mathbb Z^n$ its homogeneous part $u_m$ is also
separately $(\alpha, \beta)$ - harmonic.
\end{lemma}
\begin{proof}
Using \eqref{zetamod1} we obtain, for every $1 \leq j \leq n$:
\begin{align*}
L_{\alpha_j, \beta_j} u_m(z)  = & L_{\alpha_j, \beta_j} \int_{\mathbb T^n} \overline\zeta^m (\zeta \cdot u)(z) dm_n(\zeta)  = 
\int_{\mathbb T^n} \overline\zeta^m  L_{\alpha_j, \beta_j} (\zeta \cdot u)(z) dm_n(\zeta) \\
= & 0.
\end{align*}
\end{proof}

\begin{proposition}[\cite{KO}, Theorem 2.4 for $n=1$]\label{Reprmhom}
Assume that $u \in sh_{\alpha, \beta} (\mathbb D^n)$ is $m$ - homogeneous for some 
$m \in \mathbb Z^n$. Then, for some constant $c$, we have
\begin{equation}\label{reprmhom}
u(z) = c \mathcal F_{m^+, m^-} (z) z^{(m)} =
c \prod_{j=1}^n F_{m_j^+, m_j^-} (\vert z_j \vert^2) z_j^{(m_j)}, \qquad z \in \mathbb D^n.
\end{equation}
Conversely, every function $u(z)$ of the above form is $m$ - homogeneous and separately $(\alpha, \beta)$ - harmonic.
\end{proposition}

\begin{proof}
We proceed by induction on $n$. The case $n=1$ is Theorem 2.4 from \cite{KO}. Now let $n > 1$ and assume
that proposition holds for $n-1$.
Let us fix $z' \in \mathbb D^{n-1}$.
The function $U_{z'} (z_n) = u(z', z_n)$ 
is  $(\alpha_n, \beta_n)$ - harmonic. 
Since the theorem holds for $n=1$ we have $u(z) = v(z') F_{m_n^+, m_n^-} (\vert z_n \vert^2) z_n^{(m_n)}$. Clearly $v$ is
$(\alpha', \beta')$ - harmonic in $\mathbb D^{n-1}$ and $m' = (m_1, \ldots, m_{n-1})$ - homogeneous. Therefore inductive
hypothesis provides us representation \eqref{reprmhom}. The converse follows from the fact that each of the functions 
$F_{m_j^+, m_j^-} (\vert z_j \vert^2) z_j^{(m_j)}$ is $(\alpha_j, \beta_j)$ - harmonic in $\mathbb D$.
\end{proof}

{\em Proof of Theorem \ref{expn}.}

The function $u$ is in $C^\infty (\mathbb D^n)$ by Lemma \ref{cinfty}, hence Proposition \ref{Fouriervector} gives us expansion $u(z) = \sum_{m \in \mathbb Z^n} u_m(z)$  of $u$ into its $m$ - th
homogeneous components $u_m$, where convergence is in the topology of $C^\infty (\mathbb D^n)$. By 
Lemma \ref{alphabetamhom} each $u_m$ is separately $(\alpha, \beta)$ - harmonic. Next, Proposition \ref{Reprmhom} gives
representation of each $u_m$ in the form given in \eqref{reprmhom}. Therefore we have
\begin{equation}\label{orth}
u(z)  = \sum_{m \in \mathbb Z^n} c_m  \prod_{j=1}^n  F_{m_j^+, m_j^-} (\vert z_j \vert^2) {z_j}^{(m_j)}
 =  \sum_{(t,s) \in H}  c_{t, s} \mathcal F_{t, s} (z) z^t \overline z^s, \quad z \in \mathbb D^n. 
\end{equation}
It remains to show that $c_{p, q} = \frac{\partial^{(p, q)}}{p! q!} u (0)$ for each pure multiindex 
$(p, q) $ in $\mathbb Z^n_+ \times \mathbb Z^n_+$.  Fix $(p, q)$ in $H$, set $k = \vert p \vert + \vert q \vert$ and
consider the Taylor expansion of $u$ at zero:
\begin{equation}\label{taylor0}
u(z) = \sum_{j=0}^k   \sum_{\vert t \vert + \vert s \vert = j} \frac{\partial^{(t, s)}u(0)}{t! s!} z^t \overline z^s +
O(\| z \|^{k+1}).
\end{equation}
Using \eqref{taylor0} and orthogonality relations for the standard basis for $L^2(\mathbb T^n)$ we obtain
\begin{equation}\label{rto0}
I(r) = \int_{\mathbb T^n} u_r(\zeta) \overline\zeta^p \zeta^q dm_n(\zeta) = \frac{\partial^{(p, q)}u(0)}{p! q!}r^k
+ O(r^{k+1}).
\end{equation}
On the other hand, $F(a, b; c; 0) = 1$ and hence  $\mathcal F_{t, s} (0) = 1$ for all $(s, t)$ in $H$. 
Setting $z = r\zeta$ in \eqref{orth} and using the same orthogonality relations we obtain
$$I(r) = c_{p,q}r^k\mathcal F_{p,q} (r, r \ldots, r) .$$ 
This combined with \eqref{rto0} gives desired formula for $c_{p,q}$. $\Box$

\section{$H^p$ theory for separately $(\alpha, \beta)$ - harmonic functions}\label{HpSec}

\begin{definition}\label{shinfty}
We define 
$sh_{\alpha, \beta}^\infty (\mathbb D^n) = sh_{\alpha, \beta}(\mathbb D^n) \cap L^\infty (\mathbb D^n)$ and  
$$Csh_{\alpha, \beta}(\mathbb D^n) = \{ u \in C(\overline{\mathbb D^n}) : u \vert_{\mathbb D^n} \in sh_{\alpha, \beta}
(\mathbb D^n) \}.
$$
\end{definition}
Clearly $Csh_{\alpha, \beta}(\mathbb D^n)  \subset sh_{\alpha, \beta}^\infty (\mathbb D^n)$ and 
both spaces are Banach spaces with respect to the supremum norm. Using conditions $\Re \alpha_j + \Re \beta_j > -1$ and
Proposition \ref{Reprmhom} we deduce that for each $(p, q)$ in $H$ the function
\begin{equation}\label{phipq}
\varphi_{p,q}(z) = \mathcal F_{p,q} (z) z^p \overline z^q, \qquad z \in \overline{\mathbb D}^n
\end{equation}
belongs to the space  $Csh_{\alpha, \beta}(\mathbb D^n)$.
 

\begin{theorem}\label{Hinfty}
If $\varphi \in L^\infty(\mathbb T^n)$, then $P_{\alpha,\beta}[\varphi] \in sh_{\alpha, \beta}^\infty (\mathbb D^n) $. Moreover,
\begin{equation}\label{supest}
\| P_{\alpha, \beta}[\varphi] \|_{L^\infty (\mathbb D^n)} \leq K(\alpha, \beta) \| \varphi \|_{L^\infty (\mathbb T^n)}.
\end{equation}
In particular, $P_{\alpha, \beta}$ is a bounded linear
operator from $L^\infty(\mathbb T^n)$ to $sh_{\alpha, \beta}^\infty  (\mathbb D^n)$.
\end{theorem}
\begin{proof}
By Proposition \ref{shPofmeasures} the function $P_{\alpha, \beta}[\varphi]$ is in $sh_{\alpha, \beta}(\mathbb D^n)$. The
estimate \eqref{supest} follows from \eqref{L1p} and \eqref{poissonL1}.
\end{proof} 	

\begin{lemma}\label{slices}
Assume $u \in Csh_{\alpha, \beta}(\mathbb D^n)$, $1 \leq k < n$ and $\zeta_j \in \mathbb T$, $k < j \leq n$.
Set $\alpha' = (\alpha_1, \ldots, \alpha_k)$, 
$\beta' = (\beta_1, \ldots, \beta_k)$. Then the function
$$u_{\zeta_{k+1}, \ldots, \zeta_n}(z_1, \ldots, z_k) = u(z_1, \ldots, z_k, \zeta_{k+1}, \ldots, \zeta_n)$$ 
belongs to the space $Csh_{\alpha', \beta'} (\mathbb D^k)$.
\end{lemma}
\begin{proof}
Choose a sequence of points $w_l = (w_{k+1}^l, \ldots, w_n^l)$ in $\mathbb D^{n-k}$ which converges to 
$\zeta = (\zeta_{k+1}, \ldots, \zeta_n)$ and set $u_l(z_1, \ldots, z_k) = u(z_1, \ldots, z_k, w_{k+1}^l, \ldots, w_n^l)$. Then
$u_l$ uniformly converges to $u_{\zeta_{k+1}, \ldots, \zeta_n}$ on $\overline{\mathbb D^k}$. Clearly 
$u_{\zeta_{k+1}, \ldots, \zeta_n}$ is continuous on $\overline{\mathbb D^k}$. Since $L_{\alpha_j, \beta_j} u_l = 0$ on
$\mathbb D^k$, the mentioned uniform convergence gives us that $L_{\alpha_j, \beta_j}u_{\zeta_{k+1}, \ldots, \zeta_n} = 0$
in the weak sense, $1 \leq j \leq k$. Then we have $L u_{\zeta_{k+1}, \ldots, \zeta_n} = 0$ in the weak sense as well, where
$L = L_{\alpha_1, \beta_1} + \cdots + L_{\alpha_k, \beta_k}$. Since $L$ is elliptic, elliptic regularity gives us 
$u_{\zeta_{k+1}, \ldots, \zeta_n} \in C^\infty (\mathbb D^k)$ and therefore 
$L_{\alpha_j, \beta_j}u_{\zeta_{k+1}, \ldots, \zeta_n} = 0$ in the classical sense for $1 \leq j \leq k$. This implies
$u_{\zeta_{k+1}, \ldots, \zeta_n} \in Csh_{\alpha', \beta'} (\mathbb D^k)$.
\end{proof}
Of course, any other $n-k$ variables belonging to $\mathbb T$ could have been fixed, with analogous conclusion.

\begin{theorem}\label{nepgran}
If $u \in Csh_{\alpha, \beta}(\mathbb D^n)$, then 
$u (z)= P_{\alpha, \beta}[u\vert_{\mathbb T^n}] (z)$ for all $z \in \mathbb D^n$.
\end{theorem}
	
\begin{proof}
Let $z = (z', z_n) \in \mathbb D^n$ set $g_{z'}(\lambda)= u(z', \lambda)$, $\vert \lambda \vert \leq 1$. 
Clearly $g_{z'}$ is continuous on $\overline{\mathbb D}$ and $ (\alpha_n,\beta_n)$ -  harmonic in $ \mathbb D$.
Then, by Theorem \ref{KODP} we get
$$ u(z', z_n)=g_{z'}(z_n) = \int_{\mathbb T} P_{\alpha_n, \beta_n}(z_n, \zeta_n) u(z', \zeta_n) dm_1  (\zeta_n).$$ 
Since, $u_{\zeta_n}(z_1, \ldots, z_{n-1}) = u(z_1, \ldots, z_{n-1}, \zeta_n)$ is
$((\alpha_1, \ldots, \alpha_{n-1}), (\beta_1, \ldots, \beta_{n-1}))$ - harmonic  for every $\zeta_n$ in $\mathbb T$, by Lemma \ref{slices}, we analogously obtain
\begin{align*}
u(z) = & \int_{\mathbb T} P_{\alpha_n, \beta_n}(z_n, \zeta_n) \\
& \left( \int_{\mathbb T} 
P_{\alpha_{n-1}, \beta_{n-1}}(z_{n-1}, \zeta_{n-1}) u(z_1, \ldots, z_{n-2},\zeta_{n-1}, \zeta_n) dm_1(\zeta_{n-1}) \right) dm_1(\zeta_n).
\end{align*}
Clearly we can iterate this process, using Lemma \ref{slices} at each step and obtain $n$-times iterated integral, which is
equal to $ P_{\alpha, \beta}[u\vert_{\mathbb T^n}] (z)$ by Fubini's theorem.
\end{proof}

\begin{theorem}\label{nepprod}
If $\psi \in C({\mathbb T^n})$ then $P_{\alpha, \beta}[\psi]$ has continuous extension $u$ on $\overline{\mathbb D^n}$,
where $u \in Csh_{\alpha, \beta}(\mathbb D^n)$ and $u \vert_{\mathbb T^n} = \psi$.
\end{theorem}

\begin{proof}
The set $Y$ of complex valued functions on $\mathbb D^n$ which extend continuously to $\mathbb D^n$ is
a closed vector subspace of the Banach space $X= CB(\mathbb D^n)$ of continuous bounded functions on $\mathbb D^n$, with supremum norm. Let $A$ be the set of all finite sums of the functions of the form
\begin{equation}\label{algebra}
\psi (\zeta) = \psi_1 (\zeta_1) \psi_2 (\zeta_2) \cdots \psi_n(\zeta_n), \qquad \zeta = (\zeta_1, \ldots, \zeta_n) \in
\mathbb T^n,
\end{equation}
where $\psi_j$ is a continuous complex valued function on $\mathbb T$ for $1 \leq j \leq n$. It is easy to verify that $A$ is an
algebra of continuous functions on a compact $\mathbb T^n$ which is closed under conjugation, separates points on
$\mathbb T^n$ and contains constants. Hence, by the Stone - Weierstrass theorem, it is dense in $C(\mathbb T^n)$. By
\eqref{prodpoisson} we have, for $\psi$ as in \eqref{algebra}:
$$P_{\alpha, \beta}[\psi] (z_1, \ldots, z_n) = \prod_{j=1}^n (P_{\alpha_j, \beta_j}[\psi_j])(z_j).$$
Using Theorem \ref{KODP}, this shows that $P_{\alpha, \beta}$ maps $A$ into $Y$ and that restriction of 
$P_{\alpha, \beta}[\psi]$ to $\mathbb T^n$ is equal to $\psi$, for $\psi \in A$.  Moreover, $P_{\alpha, \beta} :
C(\mathbb T^n) \to CB(\mathbb D^n)$ is continuous by Theorem \ref{Hinfty}. Since $A$ is dense in $C(\mathbb T^n)$
and $Y$ is closed in $X$, it follows that $P_{\alpha, \beta}$ maps $C(\mathbb T^n)$ into $Y$, which completes the proof.
\end{proof}

Theorems \ref{nepprod} and \ref{nepgran} imply that $C(\mathbb T^n)$ and $Csh_{\alpha, \beta} (\mathbb D^n)$
are isomorphic as Banach spaces, the operator $P_{\alpha, \beta} :C(\mathbb T^n) \rightarrow Csh_{\alpha, \beta}
(\mathbb D^n)$ provides an isomorphism (not necessarily isometric). 

Another consequence of Theorems \ref{nepprod} and \ref{nepgran} is the following result on the Dirichlet problem for our
system of PDEs:

\begin{proposition}\label{Dirchndim}
For every $\varphi \in C(\mathbb T^n)$ the Dirichlet problem
 \begin{equation*}\label{Dirprposed}
\begin{cases}
L_{\alpha_j, \beta_j} u = 0, & z \in \mathbb D^n, \quad 1 \leq j \leq n  \\
u(z) = \varphi (z), & z \in \mathbb T^n.
\end{cases}
\end{equation*}
has unique solution $u$ in the space $Csh_{\alpha, \beta} (\mathbb D^n)$ and $u(z) = P_{\alpha, \beta}[\varphi](z)$ for
$z \in \mathbb D^n$.
\end{proposition} 

We can apply Proposition \ref{Dirchndim} to the function $\varphi_{p, q} (z) = \mathcal F_{p,q} (z)  z^p \overline z^q$, which
is in the space $Csh_{\alpha, \beta} (\mathbb D^n)$. Since 
$\varphi_{p,q} (\zeta) = \mathcal F_{p,q}({\bf 1}) \zeta^p \overline \zeta^q$ for $\zeta \in \mathbb T^n$ we get
\begin{equation}\label{phiext}
P_{\alpha, \beta}[\zeta^p \overline \zeta^q] (z) = \frac{\mathcal F_{p,q} (z)}{\mathcal F_{p,q}({\bf 1})} z^p \overline z^q,
\qquad z \in \mathbb D^n.
\end{equation}
We combine the above formula with Theorem \ref{Hinfty} and obvious equality $\| \zeta^p \overline \zeta^q \|_{L^\infty
(\mathbb T^n)} = 1$ to deduce the following result which is, in the particular case $n = 1$, an estimate of a
family of hypergeometric functions.

\begin{proposition}\label{Unesthyp}
For every $(p, q) \in H$ we have the following estimate:
\begin{equation}\label{unesthyp}
\left\vert  \frac{\mathcal F_{p,q} (z)}{\mathcal F_{p,q}({\bf 1})} z^p \overline z^q \right\vert \leq K(\alpha, \beta), 
\qquad z \in \overline{\mathbb D}^n.
\end{equation}
\end{proposition}

We have already noted that $u(rz)$ is not necessarily separately $(\alpha, \beta)$ - harmonic if $u$ is separately 
$(\alpha, \beta)$ - harmonic, see \eqref{zetar}. This leads to difficulties
in applying standard arguments on integral representations which involve Banach - Alaouglu theorem. The next theorem
allows us to surmount these difficulties, see proofs of Theorems \ref{Repr1} and \ref{ReprLp}.

\begin{theorem}\label{LimitPoisson}
If $u$ is separately $(\alpha, \beta)-$ harmonic in $\mathbb D^n$, then
\begin{equation}\label{limitPoisson}
\lim_{{\bf r \to 1}} P_{\alpha, \beta}[u_{\bf r}] (z) = u(z), \qquad z \in \mathbb D^n.
\end{equation}
\end{theorem}

\begin{proof}
We recall  locally uniform expansion of $u$ from Theorem \ref{expn}
\begin{equation}\label{exp1applied1}
u(z) = \sum_{(p,q) \in H} c_{p, q} \mathcal F_{p,q} (z) z^p \overline z^q, 
\qquad z \in \mathbb D^n, \quad c_{p, q} = \frac{ \partial^{(p, q)}u}{p!q!} (0).
\end{equation}
Therefore we have, using \eqref{phiext}:
\begin{align}\label{limr}
\lim_{{\bf r \to 1}} P_{\alpha, \beta}[u_{\bf r}] (z) &  = \lim_{{\bf r \to 1}} 
 \int_{\mathbb T^n} P_{\alpha, \beta} (z, \zeta) u ( {\bf r} \cdot  \zeta) dm_n(\zeta)  \nonumber    \\
& =  \lim_{{\bf r \to 1}} \int_{\mathbb T^n} P_{\alpha, \beta} (z, \zeta) 
\sum_{(p, q) \in H} c_{p, q} {\bf r}^p {\bf r}^q
\mathcal F_{p, q} ({\bf r})  \zeta^p \overline \zeta^q dm_n(\zeta) \nonumber     \\
 & =   \lim_{{\bf r \to 1}}  \sum_{(p, q) \in H}  c_{p, q} {\bf r}^p {\bf r}^q
\mathcal F_{p, q} ({\bf r})  P_{\alpha, \beta}[\zeta^p \overline \zeta^q](z)  \nonumber    \\
& =   \lim_{{\bf r \to 1}}  \sum_{(p, q) \in H} c_{p, q}  {\bf r}^p {\bf r}^q  
 \frac{\mathcal F_{p,q} ({\bf r})}{\mathcal F_{p,q}({\bf 1})}  \mathcal F_{p,q} (z)  z^p \overline z^q . 
\end{align}
Next we estimate each term $A(p, q, {\bf r}, z)$ in the above series using \eqref{estcomp} and \eqref{unesthyp}:
\begin{align*}
\vert A(p, q, {\bf r}, z) \vert & \leq C_z 
\frac{\mathcal F_{p, q} ({\bf r})}{\mathcal F_{p, q} ({\bf 1})}
{\bf r}^p  {\bf r}^q   \prod_{j=1}^n p_j^{-2} q_j^{-2}   \leq C_z K(\alpha, \beta)  \prod_{j=1}^n p_j^{-2} q_j^{-2},
\end{align*}
where we follow earlier convention of interpreting $0^{-2}$ as $1$. Since the multiple series $\sum_H \prod_{j=1}^n
p_j^{-2} q_j^{-2}$ converges, and the terms do not depend on ${\bf r}$, the interchange of limit and summation in
\eqref{limr} is legitimate. This ends the proof.
\end{proof}

Next we derive a local version of Theorem \ref{nepprod}.

\begin{theorem}\label{localcont}
Assume $\varphi : \mathbb T^n \to \mathbb C$  is  continuous at 
$\zeta^0  = (e^{i\theta^0_1}, \ldots, e^{i\theta^0_n}) $ in $ \mathbb T^n$ and belongs to $L^\infty(\mathbb T^n)$. Then $u = P_{\alpha, \beta} [\varphi]$ has continuous extension to $\mathbb D^n \cup \{ \zeta^0 \}$.
\end{theorem}

\begin{proof}
Let us consider an auxilliary function $\varphi (\zeta) - \varphi (\zeta^0)$. An application of Theorem \ref{nepprod} to the constant function $\psi (\zeta) = \varphi (\zeta^0)$ shows that we can assume, without loss of generality, that
$\varphi (\zeta^0) = 0$. For $z = (r_1 e^{i \theta_1}, \ldots, r_n e^{i \theta_n})$ in $ \mathbb D^n$  we have, see \eqref{poissonL1}:
\begin{align*}
\vert u(z) \vert \leq & \int_{\mathbb T^n}  \vert \varphi (\zeta) \vert
\prod_{j = 1}^n \vert P_{\alpha_j, \beta_j}(r_j e^{i \theta_j}, \zeta_j) \vert dm_n(\zeta) \\
= & (2\pi)^{-n} \int_{-\pi}^\pi \cdots  \int_{-\pi}^\pi \mid \varphi (e^{it_1}, \ldots, e^{it_n}) \mid  \prod_{j=1}^n
\mid v_{\alpha_j, \beta_j} (r_j e^{i(\theta_j - t_j)})  \mid dt_1 \ldots dt_n
\end{align*}
For $0 < \delta < \pi$ and $\theta = (\theta_1, \ldots, \theta_n)$ in $[-\pi, \pi]^n$ we write 
$$M_\infty (\varphi, \theta, \delta) = \sup_{\vert \theta_j^\prime - \theta_j \vert < \delta} \vert \varphi (e^{i\theta_1^\prime},
\ldots, e^{i \theta_n^\prime}) \vert.$$
The complement of $[-\delta, \delta]^n$ in $[-\pi, \pi]^n$ is contained in the union of the sets 
$$S_j = \{ \theta \in [-\pi, \pi ]^n : \vert \theta_j \vert \geq \delta \}, \qquad 1 \leq j \leq n.$$
Thus we have $\vert u(z) \vert \leq I_0(z, \delta) +  I_1(z, \delta)  + I_2(z, \delta) + \cdots + I_n(z, \delta)$ where
\begin{align*}
I_0 (z, \delta)  = & \frac{1}{(2\pi)^n} \int_{[-\delta, \delta]^n} 
\vert \varphi (e^{i (\theta_1 - t_1)}, \ldots, e^{i (\theta_n - t_n)}) 
\vert \prod_{j=1}^n \vert v_{\alpha_j, \beta_j}(r_j e^{i t_j}) \vert dt_1\ldots dt_n \\
\leq &  \frac{ M_\infty (\varphi, \theta, \delta)}{(2\pi)^n} \int_{[-\delta, \delta]^n} \prod_{j=1}^n \vert v_{\alpha_j, \beta_j}(r_j e^{i t_j}) \vert 
dt_1 \ldots dt_n 
\leq   K(\alpha, \beta) M_\infty (\varphi, \theta, \delta)
\end{align*} 
and, for $1 \leq j \leq n$:
\begin{align*}
I_j(z, \delta) =  \frac{1}{(2\pi)^n} & \int_{S_j} \vert \varphi (e^{i (\theta_1 - t_1)}, \ldots, e^{i (\theta_n - t_n)}) \vert 
 \prod_{k=1}^n \vert  v_{\alpha_k, \beta_k} (r_k e^{i t_k})  \vert  dt_1 \ldots dt_n   \\
\leq &   \frac{\| \varphi \|_\infty}{(2\pi)^n}                   \prod_{k \not= j} \int_{-\pi}^\pi 
\vert v_{\alpha_k, \beta_k} (r_k e^{i t_k})  \vert dt_k 
\cdot \int_{\delta < \vert t_j \vert \leq \pi} \vert v_{\alpha_j, \beta_j}(r_j e^{i t_j}) \vert dt_j   \\
\leq & \frac{K_j(\alpha, \beta)}{2\pi}  \| \varphi \|_\infty 
\sup_{\delta < \vert t_j \vert \leq \pi} \vert  v_{\alpha_j, \beta_j} (r_j e^{i t_j}) \vert  \\
\leq & C_j(\delta)   K_j(\alpha, \beta)(1 - r_j^2)^{\Re \alpha_j + \Re \beta_j + 1}  \| \varphi \|_\infty.
\end{align*}
Hence, setting $C_{\alpha, \beta} (\delta) = \max_j C_j(\delta) K_j(\alpha, \beta)$, we have obtained
\begin{equation}\label{i0in}
\vert u(z) \vert \leq  K(\alpha, \beta) M_\infty (\varphi, \theta, \delta) + C_{\alpha, \beta} (\delta) \| \varphi \|_\infty
\sum_{j=1}^n (1 - r_j^2)^{\Re \alpha_j + \Re \beta_j + 1}.
\end{equation}
Given $\varepsilon > 0$ we can choose $\delta > 0$ such that $\mid \theta_j - \theta_j^0 \mid < 2\delta$, $1 \leq j \leq n$,
implies $\mid \varphi (\zeta) \mid < \varepsilon$, where $\zeta = (e^{i\theta_1}, \ldots, e^{i\theta_n})$. Therefore, if
$\mid \theta_j - \theta_j^0 \mid < \delta$, $1 \leq j \leq n$, then $ M_\infty (\varphi, \theta, \delta) \leq \varepsilon$. Now the
result follows from \eqref{i0in}, conditions $\Re \alpha_j + \Re \beta_j > -1$ and observation that $z = (r_1 e^{i \theta_1},
\ldots, r_n e^{i \theta_n})$ converges to $\zeta^0$ if and only if $r_j \to 1$ for all $1 \leq j \leq n$ and $\theta_j \to
\theta_j^0$ for all $1 \leq j \leq n$.
\end{proof}

\begin{definition}\label{shLp}
Let $1 \leq p < \infty$. We define the space $sh_{\alpha, \beta}^p({\mathbb D^n})$ as the set of all $u \in sh_{\alpha, \beta}
(\mathbb D^n)$ such that
\begin{equation}\label{shpnorm}
\| u \|_{\alpha, \beta; p} = \sup_{{\bf r} \in [0,1)^n}     \left( \int_{\mathbb T^n} 
\vert u_{\bf r}  (\zeta) \vert^p dm_n (\zeta) \right)^\frac{1}{p} < \infty.
\end{equation}
\end{definition}
It is easy to verify that the spaces $sh_{\alpha, \beta}^p (\mathbb D^n)$ are Banach spaces with respect to the above norm.

\begin{proposition}\label{gm}
If $\mu \in \mathcal M (\mathbb T^n)$ then $u = P_{\alpha,\beta}[d\mu]$ belongs to
$sh_{\alpha, \beta}^1 (\mathbb D^n)$ and $\| u \|_{\alpha, \beta;1} \leq K(\alpha, \beta) \| \mu \|$.
\end{proposition}

\begin{proof}
By Proposition  \ref{shPofmeasures} $u$ is separately $(\alpha, \beta)$ - harmonic and  by \eqref{convPmu2} and \eqref{L1V} we have 
$$\| u_{\bf r} \|_1 = \| (V_{\alpha, \beta})_{\bf r} \ast d \mu \|_1 \leq \| (V_{\alpha, \beta})_{\bf r}
\|_1 \| \mu \| \leq K(\alpha, \beta) \| \mu \|, \qquad {\bf r} \in [0, 1)^n.$$
\end{proof}
We proved that $P_{\alpha, \beta}$ is a bounded linear operator from $\mathcal M (\mathbb T^n)$ to 
$sh_{\alpha, \beta}^1 (\mathbb D^n)$, with norm bounded by $K(\alpha, \beta)$.

The following lemma will be useful in arguments involving duality.

\begin{lemma}\label{dualitylemma}
Assume that $\mu$ and $\nu$ are complex Borel measures on $\mathbb T^n$, $u = P_{\alpha, \beta}[d\mu]$,
$v = P_{\beta, \alpha}[d\nu]$ and let ${\bf r} \in [0, 1)^n$. Then
$\langle u_{\bf r}, d\nu \rangle = \langle v_{\bf r}, d\mu \rangle$, or explicitly
\begin{equation}\label{dualityeq}
\int_{\mathbb T^n} (P_{\alpha, \beta}[d\mu]) ({\bf r} \cdot \zeta) d\nu(\zeta) = 
\int_{\mathbb T^n} (P_{\beta, \alpha}[d\nu]) ({\bf r} \cdot \xi) d\mu(\xi).
\end{equation}
\end{lemma}

\begin{proof}
Since $P_{\alpha, \beta}({\bf r} \cdot \zeta, \xi) = P_{\beta,\alpha}({\bf r} \cdot \xi, \zeta)$ we have
\begin{align*} 
\int_{\mathbb T^n} (P_{\alpha, \beta}[d\mu]) ({\bf r} \cdot \zeta) d\nu(\zeta)  & = 
\int_{\mathbb T^n} \left( \int_{\mathbb T^n} P_{\alpha, \beta} ({\bf  r} \cdot \zeta, \xi) d\mu (\xi) \right) d\nu(\zeta) \\
& = \int_{\mathbb T^n} \left( \int_{\mathbb T^n} P_{\beta, \alpha} ({\bf r} \cdot \xi, \zeta) d\nu (\zeta)
\right) d\mu (\xi) \\
& = \int_{\mathbb T^n} (P_{\beta, \alpha}[d\nu]) ({\bf r} \cdot \xi) d\mu (\xi).
\end{align*}
\end{proof}

It was proved in \cite{Seur}  that if $ f (z) = P[d\mu](z)$ where $\mu \in \mathcal M (\mathbb T^n)$, then 
$ f_r dm \rightarrow d \mu $
weak$^\ast$ as $r \rightarrow 1$ where $P$ is the Poisson kernel for $n$ - harmonic functions (the case $\alpha_j =
\beta_j = 0$). The next theorem generalizes this result to separately $(\alpha, \beta)$ - harmonic functions, at the same
time allowing for more general limits.

\begin{theorem}\label{Cweakstar}
Let $u = P_{\alpha,\beta}[d\mu]$ where $\mu \in \mathcal M (\mathbb T^n)$. Then
$u_{\bf r} dm_n$ converges weak$^\ast$ to $d\mu$ as ${\bf r} = (r_1, \ldots, r_n) \to (1, \ldots, 1) = {\bf 1}$ in $[0, 1)^n$.
\end{theorem}

\begin{proof}
We fix $\varphi \in C(\mathbb T^n)$ and set $v = P_{\beta, \alpha}[\varphi]$. Using Lemma \ref{dualitylemma} with 
$d\nu = \varphi dm_n$ we obtain
$$
\int_{\mathbb T^n} u({\bf r} \cdot \zeta) \varphi (\zeta) dm_n(\zeta) = \int_{\mathbb T^n} v({\bf r} \cdot \xi) d\mu (\xi).
$$
However, by Theorem \ref{nepprod}, $v({\bf r} \cdot \xi)$ converges uniformly on $\mathbb T^n$ to $\varphi  (\xi)$ as 
$\bf r$ converges to $\bf 1$. Hence, $\lim_{{\bf r} \to {\bf 1}} \int_{\mathbb T^n} u({\bf r} \cdot \zeta) \varphi (\zeta) dm_n(\zeta) = \int_{\mathbb T^n} \varphi (\xi) d\mu (\xi)$.
\end{proof}

\begin{theorem}\label{LPnorm}
If $\psi \in L^p(\mathbb T^n)$ for some $1 \leq p \leq \infty$ then $P_{\alpha, \beta}[\psi]$ is in the space
$sh_{\alpha, \beta}^p (\mathbb D^n)$ and
\begin{equation}\label{urp}
 \| P_{\alpha, \beta}[\psi] \|_{\alpha, \beta ;p} \leq K(\alpha, \beta) \|\psi \|_p.
 \end{equation}
For $1 \leq p < \infty$ we have 
\begin{equation}\label{Lpconv}
\lim_{\bf r \to 1} \| (P_{\alpha, \beta}[\psi])_{\bf r} - \psi \|_{L^p(\mathbb T^n)} = 0, \qquad \psi \in L^p (\mathbb T^n).
\end{equation}
\end{theorem}

\begin{proof}
In proving \eqref{urp} we argue as in Proposition \ref{gm} and use Young's inequality for convolutions.
Indeed,  $u = P_{\alpha, \beta}[\psi]$ is separately 
$(\alpha, \beta)$ - harmonic by Proposition \ref{shPofmeasures} and  by \eqref{convPf} and \eqref{L1V} we have 
$$\| u_{\bf r} \|_p = \| (V_{\alpha, \beta})_{\bf r} \ast f \|_p \leq \| (V_{\alpha, \beta})_{\bf r}
\|_1 \| f \|_p \leq K(\alpha, \beta) \| f \|_p, \qquad {\bf r} \in [0, 1)^n.$$

Now we assume $1 \leq p < \infty$. Let $\psi \in L^p(\mathbb T^n)$ and let $\varepsilon > 0$. Since
$C(\mathbb T^n) $ is dense in $L^p(\mathbb T^n)$ we can choose $\varphi \in C(\mathbb T^n)$ such that 
$\|\varphi-\psi \|_p  < \varepsilon $.  Then, by \eqref{urp}:
\begin{align*}
\|  (P_{\alpha, \beta}[\psi])_{\bf r} - \psi\|_p & \leq  
\|  (P_{\alpha, \beta}[\psi - \varphi])_{\bf r} \|_p +
\|  (P_{\alpha, \beta}[\varphi])_{\bf r}    -     \varphi \|_p + \| \varphi-\psi \|_p \\
& \leq  [K(\alpha, \beta) + 1] \varepsilon + \|  (P_{\alpha, \beta}[\varphi])_{\bf r}    -     \varphi \|_p \\
&  \leq  [K(\alpha, \beta) + 1] \varepsilon + \|  (P_{\alpha, \beta}[\varphi])_{\bf r}    -     \varphi \|_{C(\mathbb T^n)}.
\end{align*}
Since, by Theorem \ref{nepprod}, $\lim_{\bf r \rightarrow 1} 
\| (P_{\alpha, \beta}[\varphi])_{\bf r} - \varphi \|_{C(\mathbb T^n)} = 0$, we obtain from the above estimates
$\limsup_{\bf r \rightarrow 1}  \|  (P_{\alpha, \beta}[\psi])_{\bf r} - \psi\|_p  \leq  [K(\alpha, \beta) + 1] \varepsilon$ which
gives \eqref{Lpconv}.
\end{proof}

\begin{theorem}\label{L1weakstar}
Let $u = P_{\alpha,\beta}[f]$ where $f \in L^\infty (\mathbb T^n)$. Then
$u_{\bf r}$ converges weak$^\ast$ to $f$ as ${\bf r} = (r_1, \ldots, r_n) \to (1, \ldots, 1) = {\bf 1}$ in $[0, 1)^n$.
\end{theorem}
\begin{proof}
We fix $\varphi \in L^1(\mathbb T^n)$ and set $v = P_{\beta, \alpha}[\varphi]$. Using Lemma \ref{dualitylemma} with 
$d\nu = \varphi dm_n$ and $d\mu = f dm_n$ we obtain
$$
\int_{\mathbb T^n} u({\bf r} \cdot \zeta) \varphi (\zeta) dm_n(\zeta) = 
\int_{\mathbb T^n} v({\bf r} \cdot \xi) f(\xi) dm_n (\xi).
$$
However, by Theorem \ref{LPnorm}, $v({\bf r} \cdot \xi)$ converges in $L^1(\mathbb T^n)$ norm to $\varphi  (\xi)$ as 
$\bf r$ tends to $\bf 1$. Hence, $\lim_{{\bf r} \to {\bf 1}} \int_{\mathbb T^n} u({\bf r} \cdot \zeta) \varphi (\zeta) dm_n(\zeta) = \int_{\mathbb T^n} f(\xi) \varphi (\xi) dm_n (\xi)$.

\end{proof}

Next we turn to integral representation theorems for functions in $sh_{\alpha, \beta}^p(\mathbb D^n)$ spaces.

\begin{theorem}\label{Repr1}
Assume $u$ is separately $(\alpha, \beta)$ - harmonic function on $\mathbb D^n$ such that
\begin{equation}\label{sh1r} 
\sup_{0 \leq r <1} \int_{\mathbb T^n} \mid u(r \zeta) \mid dm_n(\zeta) < \infty.
\end{equation}
Then there is a unique $ \mu \in \mathcal{M}(\mathbb{T}^n)$ such that $u = P_{\alpha, \beta}[d\mu]$. 
\end{theorem}

\begin{proof}
Since a sequence $u((1-1/k)\zeta) dm_n(\zeta)$ of complex Borel measures on $\mathbb T^n$ is bounded in $\mathcal M
(\mathbb T^n) = C(\mathbb T^n)^\ast$ by \eqref{sh1r}, Banach - Alaoglu theorem and separability of $C(\mathbb T^n)$ give us a weak$^\ast$ convergent subsequence. This gives us a sequence $r_l$ in $[0, 1)$ such that 
$\lim_{l \to \infty} r_l = 1$ and a complex Borel measure $\mu$ on $\mathbb T^n$ such that
\begin{equation}\label{subsequence}
\lim_{l \rightarrow \infty} \int_{\mathbb T^n} g(\zeta) u (r_l \zeta) dm_n(\zeta) =
\int_{\mathbb T^n} g(\zeta) d\mu(\zeta)  \qquad  \mbox{for all} \quad g \in C(\mathbb T^n).
\end{equation}
We apply \eqref{subsequence} to $g_z(\zeta) = P_{\alpha, \beta} (z, \zeta)$ where $z$ in $\mathbb D^n$ is fixed. This is
legitimate because $g_z$ is continuous on $\mathbb T^n$. By Theorem
\ref{LimitPoisson} we have
\begin{align*}
P_{\alpha, \beta}[d\mu] (z) & = \int_{\mathbb T^n} P_{\alpha, \beta} (z, \zeta) d\mu (\zeta) = 
\lim_{l \rightarrow \infty} \int_{\mathbb T^n} P_{\alpha, \beta} (z, \zeta) u (r_l \zeta) dm_n(\zeta) \\
& = \lim_{l \rightarrow \infty} P_{\alpha, \beta}[u_{r_l}](z)  = u(z) 
\end{align*}
Uniqueness follows from Theorem \ref{Cweakstar}.
\end{proof}

Theorem \ref{Repr1} and Proposition \ref{gm} imply that the following conditions are equivalent for a given separately
$(\alpha, \beta)$ - harmonic function $u$ on $\mathbb D^n$:
\begin{align*}
\sup_{0 \leq r <1} \int_{\mathbb T^n} \mid u(r \zeta) \mid dm_n(\zeta) < \infty, \\
 \sup_{{\bf r} \in [0,1)^n}   \int_{\mathbb T^n} 
\vert u_{\bf r}  (\zeta) \vert dm_n (\zeta)  < \infty, \\
u = P_{\alpha, \beta}[d\mu] \quad {\rm where} \quad \mu \in \mathcal M (\mathbb T^n).
\end{align*}
For $n$ - harmonic functions equivalence of the last two conditions appeared in \cite{CZ}.

\begin{theorem}\label{ReprLp}
Let $ 1 < p\leq\infty $ and assume $u \in sh_{\alpha, \beta} (\mathbb D^n)$ satisfies
\begin{equation}\label{shpr}
\sup_{0 \leq r < 1} \| u_r \|_{L^p(\mathbb T^n)} < \infty.
\end{equation}
Then there exists a unique function $ \psi \in L^p ({\mathbb T^n}) $ such that $u = P_{\alpha, \beta}[\psi]$.
\end{theorem}

\begin{proof}
Since $p > 1$,  $L^p (\mathbb T^n)$ is the dual of $ L^q ({\mathbb T^n})$, where $1/p + 1/q = 1$. As in the proof of
Theorem \ref{Repr1}, using this duality and Banach - Alaoglu theorem, we obtain a sequence $r_k$ in $[0, 1)$ converging to $1$ and a function $\psi$ in $L^p(\mathbb T^n)$ such that
\begin{equation}\label{weakstarLq}
\lim_{k \rightarrow \infty} \int_{\mathbb T^n} \varphi(\zeta)u(r_k \zeta) dm_n(\zeta) =\int_{\mathbb T^n} \varphi(\zeta)\psi(\zeta) dm_n(\zeta) \quad \mbox{for all}\quad \varphi \in L^q (\mathbb T^n).
\end{equation}
We apply \eqref{weakstarLq} to $\varphi _z (\zeta) = P(z, \zeta)$ where $z$ in $\mathbb D^n$ is fixed. Since 
$\varphi_z \in C(\mathbb T^n) \subset L^q(\mathbb T^n)$ this application is legitimate.    By Theorem
\ref{LimitPoisson} we have
\begin{align*}
P_{\alpha, \beta}[\psi] (z) & = \int_{\mathbb T^n} P_{\alpha, \beta} (z, \zeta) \psi(\zeta) dm_n (\zeta) = 
\lim_{k \rightarrow \infty} \int_{\mathbb T^n} P_{\alpha, \beta} (z, \zeta) u (r_k \zeta) dm_n(\zeta) \\
& = \lim_{k \rightarrow \infty} P_{\alpha, \beta}[u_{r_k}](z)  = u(z) .
\end{align*}

Uniqueness follows from Theorems \ref{L1weakstar} and \ref{LPnorm}.
\end{proof}

We have the following chain of continuous embeddings of Banach spaces
\begin{equation}\label{embeddingschain}
Csh_{\alpha, \beta} (\mathbb D^n) \subset sh_{\alpha, \beta}^\infty (\mathbb D^n)
 \subset sh_{\alpha, \beta}^q (\mathbb D^n)  \subset sh_{\alpha, \beta}^p (\mathbb D^n)
 \subset sh_{\alpha, \beta}^1 (\mathbb D^n),
\end{equation}
where $1 < p < q < \infty$. The above results show that these spaces are isomorphic, as Banach spaces, respectively to the spaces
\begin{equation}
C(\mathbb T^n) \subset L^\infty (\mathbb T^n) \subset L^q (\mathbb T^n) \subset L^p (\mathbb T^n) \subset
\mathcal M (\mathbb T^n),
\end{equation}
an isomorphism is provided by the operator $P_{\alpha, \beta}$.

Our next aim is to generalize Lemma \ref{slices} to the case $u \in sh^p_{\alpha, \beta} (\mathbb D^n)$. To do so,
we observe that most of the above results extend to the vector valued case, with values in a complex Banach space $X$. In fact,
Definition \ref{defalphabeta} extends to the case $u \in C^2(\mathbb D^n, X)$, Definition \ref{Poissonext} to the case of $X$ -
valued measures $\mu$ on $\mathbb T^n$ and functions $f$ in $L^1 (\mathbb T^n, X)$ and Definitions \ref{shinfty} and
\ref{shLp} have obvious generalizations which yield spaces $sh^p_{\alpha, \beta} (\mathbb D^n, X)$ for $1 \leq p \leq \infty$
and $Csh_{\alpha, \beta} (\mathbb D^n, X)$. Further, Proposition \ref{shPofmeasures} extends to the case of $X$ - valued
measures and Theorem \ref{KODP} extends to functions $\varphi \in C(\mathbb T, X)$. We need the following $L^1$ - 
variant of Theorem \ref{KODP}.

\begin{proposition}\label{NTdimone}
Assume $\varphi \in L^1(\mathbb T, X)$, and let $u = P_{\alpha, \beta}[\varphi]$. Then
\begin{equation}\label{NTdim1}
NT - \lim_{z \rightarrow \zeta} u (z) = \varphi (\zeta) \qquad \mbox{a.e.} \quad \zeta \in \mathbb T.
\end{equation}
\end{proposition}
A proof can be based on estimates on the kernel $P_{\alpha, \beta}$, use of maximal function and density of $C(\mathbb T, X)$
in $L^1(\mathbb T, X)$. In \cite{ABC} one can find the scalar valued case discussed in the more general situation of the unit
ball in $\mathbb C^n$.

In addition, the following results hold in the context of $X$ - valued functions, with essentially the same proofs: 
Theorems \ref{Hinfty} and \ref{localcont}, Proposition \ref{gm} and the first part of Theorem \ref{LPnorm}.

The following theorem is an extension of Lemma \ref{slices}.

\begin{theorem}\label{slicesLp}
Assume $\varphi : \mathbb T^n \rightarrow \mathbb C$ is  in $L^p (\mathbb T^n)$ where $1 \leq p \leq +\infty$ and let 
$u = P_{\alpha, \beta}[\varphi]$. Then: 

$1^o$. For each $z_n$ in $\mathbb D$ the function $v_{z_n} (z') = u(z', z_n)$, $z' \in \mathbb D^{n-1}$, belongs to the space
$X = sh_{\alpha', \beta'}^p  (\mathbb D^{n-1})$, where 
$\alpha' = (\alpha_1, \ldots, \alpha_{n-1})$ and $\beta' = (\beta_1, \ldots, \beta_{n-1})$.

$2^o$. There is a set $N \subset \mathbb T$ of measure zero such that for every $\zeta_n$ in $\mathbb T \setminus N$ there
exists a function $w_{\zeta_n}$ in $sh_{\alpha', \beta'}^p  (\mathbb D^{n-1})$ such that
\begin{equation}\label{NTslice}
NT - \lim_{z_n \rightarrow \zeta_n} u(z', z_n) = w_{\zeta_n} (z'),
\end{equation}
where the limit is both in $sh_{\alpha', \beta'}^p  (\mathbb D^{n-1})$ - norm and locally uniform over $z' \in
\mathbb D^{n-1}$.
\end{theorem}

\begin{proof}
There is a set $E \subset \mathbb T$ of measure zero such that for $\zeta_n \in \mathbb T \setminus E$ the function 
$\varphi_{\zeta_n} : \mathbb T^{n-1} \rightarrow \mathbb C$ defined by $\varphi_{\zeta_n}(\zeta')
 = \varphi (\zeta', \zeta_n)$ belongs to $L^p (\mathbb T^{n-1})$. 
We define $\Phi (\zeta_n) = P_{\alpha', \beta'}[\varphi_{\zeta_n}]$ for $\zeta_n \in \mathbb T \setminus E$. The mapping
$\Phi : \mathbb T \to X$ is defined almost everywhere and we have
\begin{align*}
\| \Phi \|_{L^1(\mathbb T, X)}^p  & = \left( \int_{\mathbb T} \| \Phi (\zeta_n) \|_X dm_1 (\zeta_n) \right)^p
\leq  \int_{\mathbb T} \|   P_{\alpha', \beta'}[\varphi_{\zeta_n}] \|_{\alpha', \beta' ;p}^p dm_1(\zeta_n)   \\
& \leq K^p (\alpha', \beta') \int_{\mathbb T} \| \varphi_{\zeta_n} \|^p_{L^p(\mathbb T^{n-1})} dm_1 (\zeta_n) \\
& =  K^p (\alpha', \beta') \| \varphi \|_{L^p(\mathbb T^n)}^p.
\end{align*}
We proved that $\Phi \in L^1(\mathbb T, sh_{\alpha', \beta'}^p (\mathbb D^{n-1}))$, in fact $\| \Phi \|_1 \leq 
K (\alpha', \beta') \| \varphi \|_p$. We have
\begin{equation}\label{Psiatz}
\Phi (\zeta_n) (z') = \int_{\mathbb T^{n-1}} V_{\alpha', \beta'} (z' \cdot \overline{\zeta'}) \varphi (\zeta', \zeta_n) 
dm_{n-1} (\zeta'),
\quad z' \in \mathbb D^{n-1}, \quad \zeta_n \in \mathbb T \setminus E.
\end{equation}
Let us set $U(z_n) = P_{\alpha_n, \beta_n}[\Phi] (z_n)$, $\vert z_n \vert < 1$. Then, for each $z_n \in \mathbb D$, $U(z_n)$
is in $sh_{\alpha', \beta'}^p (\mathbb D^{n-1})$. Since evaluation functionals are continuous on the space $X$, we use
\eqref{Psiatz} to obtain, for $z = (z', z_n) \in \mathbb D^n$:
\begin{align*}
U(z_n)(z') & = \left( \int_{\mathbb T}  V_{\alpha_n, \beta_n} (z_n, \overline{\zeta_n}) \Phi (\zeta_n) 
dm_1(\zeta_n)  \right) (z')  \\
& =  \int_{\mathbb T}  V_{\alpha_n, \beta_n} (z_n, \overline{\zeta_n}) \Phi (\zeta_n) (z') dm_1(\zeta_n) \\
& = \int_{\mathbb T}  V_{\alpha_n, \beta_n} (z_n, \overline{\zeta_n})  
\int_{\mathbb T^{n-1}} V_{\alpha', \beta'} (z' \cdot \overline{\zeta'}) \varphi (\zeta', \zeta_n) dm_{n-1} (\zeta') dm_1(\zeta_n) 
\\
& = u(z', z_n).
\end{align*}
Hence $v_{z_n} = U(z_n)$ and assertion $1^o$ is proved.

Part $2^o$ follows from Proposition \ref{NTdimone}, indeed one can take $w_{\zeta_n} = \Psi (\zeta_n)$ almost everywhere.
Locally uniform convergence follows from interior elliptic estimates and boundedness in 
$sh_{\alpha', \beta'}^p  (\mathbb D^{n-1})$ - norm. 
\end{proof}

Of course, in $1^o$ any variable $z_k$ could have been selected, and it is possible to iterate and fix any number of
variables. The special case of $\alpha = \beta = {\bf 0}$ appeared in \cite{Seur} as Theorem 8. For the case of analytic
functions see \cite{CZ} and \cite{Da}.

\section{Maximal functions and Fatou type theorems}\label{FatouSec}

In the previous section we studied convergence in $L^p(\mathbb T^n)$ norm, or in weak$^\ast$ topology, of $u_{\bf r}$ for $u$ in the appropriate $sh^p_{\alpha, \beta} (\mathbb D^n)$ space. In this section we study almost everywhere convergence of
$u_{\bf r}$ for a separately $(\alpha, \beta)$ - harmonic function $u$. The first result of this nature was related to Cesaro summability of double Fourier series in \cite{MZ}. In that paper authors noted that their results hold for Poisson kernels as well as for Fejer kernels,
thus obtaining a result for 2 - harmonic functions. Our approach is based on presentation given in \cite{Rud2}, where $n$ - harmonic functions are considered, see Theorem 2.3.1. from \cite{Rud2}. An important tool is a special maximal function $M_q$, it is used to obtain estimates of relevant restricted non-tangential maximal functions which lead to Fatou type theorems. Since the proof of crucial Theorem \ref{Wk11nt} is
rather involved, we split this section into subsections.

\subsection{A maximal function $M_q$}

For $\gamma$ in $\mathbb Z_+^n$ we define a $\gamma$ box as a subset $Q$ of $\mathbb T^n$ of the following form:
$$Q = I_1 \times \cdots \times I_n,\quad \mbox{ where}\quad I_j = \{ e^{i\theta} \mid a_j \leq \theta < b_j \}$$ is a semi-open arc on
$\mathbb T$ of length $0 < b_j -a_j \leq 2\pi$ and these lengths have the same ratio as $(2^{\gamma_1}, \ldots,
2^{\gamma_n})$. The center of such a box is denoted by $c(Q) = ( e^{ic_1}, \ldots, e^{ic_n})$, where $c_j = (a_j + b_j)/2$.
We denote the family of all $\gamma$ boxes by $\mathcal Q_\gamma$.
Next we define $\gamma$ - maximal function of a complex Borel measure $\mu$ on $\mathbb T^n$ by the following formula:

\begin{equation}\label{gammamax}
M_\gamma \mu (\zeta) = \sup_{Q \in \mathcal Q_\gamma, c(Q) = \zeta} \frac{\vert \mu \vert (Q)}{m_n(Q)},
\qquad \mu \in \mathcal M(\mathbb T^n), \quad \zeta \in \mathbb T^n.
\end{equation}

By considering suitable concentric boxes one shows the following estimate:
\begin{equation}\label{shrink}
M_{\gamma' + \gamma^{\prime\prime}} (\zeta) \leq 2^{\vert \gamma^{\prime\prime}\vert} M_{\gamma' }(\zeta),
\qquad \gamma', \gamma^{\prime\prime} \in \mathbb Z^n_+, \quad \zeta \in \mathbb T^n.
\end{equation}
As in the standard situation of $\mathbb R^n$ with balls instead of $\gamma$ boxes the following basic estimate holds
(see \cite{Rud2}, pp. 26-27):

\begin{equation}\label{maxgammaest}
m_n (\{ \zeta \in \mathbb T^n \mid M_\gamma \mu (\zeta) > \lambda \}) \leq 3^n \frac{ \| \mu \| }{\lambda},
\qquad \mu \in \mathcal M (\mathbb T^n), \quad \lambda > 0.
\end{equation}
The fact that the constant $3^n$ is independent of $\gamma \in \mathbb Z_+^n$ is of crucial importance. Next, for 
$0 < q < 1$, we introduce a maximal function $M_q$ by

\begin{equation}\label{mq}
M_q \mu (\zeta) = \sum_{\gamma \in \mathbb Z_+^n} q^{\vert \gamma \vert} M_\gamma \mu (\zeta), \qquad \mu \in
\mathcal M(\mathbb T^n), \quad \zeta \in \mathbb T^n.
\end{equation}


\begin{lemma}[\cite{Rud2}, page 26 for $q=1/2$]\label{G}
Let $0 < q < 1$ and let $\mu \in \mathcal M (\mathbb T^n)$. Then
\begin{equation}\label{g1}
m_n (\{ \zeta \in \mathbb T^n \mid M_q \mu (\zeta) > \lambda \})  \leq  
\frac{3^n}{(1 - \sqrt{q})^{2n}}     \frac{ \| \mu \|}{\lambda}          , \quad \lambda > 0.
\end{equation}
\end{lemma}

\begin{proof}
Set, for $\lambda > 0$ and $\gamma \in \mathbb Z_+^n$: 
$$
E(\lambda) = \{ \zeta \in \mathbb T^n \mid M_q \mu (\zeta) > \lambda \}, \qquad
E_\gamma(\lambda) = \{ \zeta \in \mathbb T^n \mid M_\gamma \mu (\zeta) > \lambda \}.
$$
Note that $ \sum_{k=0}^\infty q^{\frac k2} = \frac{1}{1-\sqrt{q}}=\eta (q)$ and
$ \sum_{\gamma \in \mathbb Z_+^n}  q^{\frac{\vert \gamma \vert}{2}}=\eta(q)^n$ for $0 < q < 1$. Hence, taking
into account \eqref{mq}, we see that $$E(\lambda) \subset \cup_{\gamma \in \mathbb Z_+^n} E_\gamma
(\eta(q)^{-n} q^{- \vert \gamma \vert/2} \lambda).$$ Therefore, using \eqref{maxgammaest}, we obtain
\begin{equation*}
m_n(E(\lambda)) \leq \sum_{\gamma \in \mathbb Z_+^n}  3^n 
\frac{\| \mu \|}{ \eta(q)^{-n} q^{- \vert \gamma \vert/2} \lambda} = 3^n \eta(q)^{2n} \| \mu \| \lambda^{-1},
\end{equation*}
which is desired estimate \eqref{g1}.
\end{proof}

\subsection{An estimate of $P_{\alpha, \beta}[d\mu]$ by $P_{\bf t}    [d\mu]$}


For real $t > -1$  and $z$ in $\mathbb D$ we set
\begin{equation}\label{utdim1}\
u_t(z) = u_{t/2, t/2} (z) = \frac{(1 - \vert z \vert^2)^{t+1}}{\vert 1 - z \vert^{t+2}},
\end{equation}
\begin{equation}\label{vtdim1}
v_t(z) = v_{t/2, t/2} (z) = \frac{\Gamma^2(\frac{t}{2} + 1)}{\Gamma(t+1)} 
\frac{(1 - \vert z \vert^2)^{t+1}}{\vert 1 - z \vert^{t+2}}.
\end{equation}
 An elementary inequality $\vert 1 - z \vert \geq \pi^{-1} \vert \theta \vert$, which is  valid for 
$0 \leq r < 1$ and  $-\pi \leq \theta \leq \pi$, implies the following estimate:
\begin{equation}\label{utest}
0 < u_t (re^{i\theta}) \leq 2^{t+1} \pi^{t+2} \frac{(1 - r)^{t+1}}{\vert \theta \vert^{t+2}}, \qquad 0 \leq r < 1, 
\quad -\pi \leq \theta \leq \pi.
\end{equation}

The kernel $v_t(z)$ appeared in \cite{OP} where the operator $L_{t/2, t/2}$ was analysed and a Fatou type theorem
was obtained in the one dimensional setting of $\mathbb D$.

For ${\bf t} = (t_1, \ldots, t_n) \in (-1, +\infty)^n$ we set
\begin{equation}\label{utdimn}
U_{\bf t} (z) = \prod_{j=1}^n u_{t_j} (z_j) , \quad 
V_{\bf t} (z) = \prod_{j=1}^n v_{t_j} (z_j), \qquad
 z \in \mathbb D^n.
\end{equation}
Now we have positive kernels $K_{\bf t} (z, \zeta) = V_{\bf t}(z \cdot \overline\zeta)$
and  $K'_{\bf t} (z, \zeta) = U_{\bf t}(z \cdot \overline\zeta)$ where $z \in \mathbb D^n$, $\zeta \in \mathbb T^n$. 
Clearly $K_{\bf t} = P_{\alpha, \beta}$ with $\alpha = \beta ={\bf t}/2$.

The next lemma follows from \eqref{poisson}, \eqref{abkernel}, \eqref{utdimn} and \eqref{utdim1}. It allows us to pass
from complex valued kernels $P_{\alpha, \beta}$ to positive ones in various estimates, see \eqref{Stolzest3} below.
\begin{lemma}\label{Muz}
Let $\mu \in \mathcal M (\mathbb T^n)$ and let $z \in \mathbb T^n$. Then we have
\begin{equation}\label{muz}
\vert P_{\alpha,\beta}[d\mu](z)\vert  \leq   k_{\alpha,\beta}    \int_{\mathbb T^n} U_{\bf t} (z \cdot \overline\zeta) 
d\vert \mu\vert (\zeta) 
= \frac{k_{\alpha, \beta}}{k_{\bf t}}  \int_{\mathbb T^n} K_{\bf t} (z, \zeta) d\vert \mu\vert (\zeta).
\end{equation}
Here ${\bf t} = (t_1, \ldots, t_n)$ where $t_j = \Re \alpha_j + \Re \beta_j$ for $1 \leq j \leq n$ and
$$ 
k_{\alpha,\beta} =\prod_{j=1}^n \vert c_{\alpha_j, \beta_j} \vert, \qquad k_{\bf t} = \prod_{j=1}^n
\frac{\Gamma^2\left(\frac{t_j}{2} + 1\right)}{\Gamma(t_j+1)}.
$$
\end{lemma}

\subsection{Estimates within Stolz cone with fixed $\bf r$}

We define the full Stolz cone at $\zeta \in \mathbb T^n$ of aperture $0 < A < \infty$ by
\begin{equation}\label{fullStolz}
S_A(\zeta) = 
\prod_{j=1}^n S_A (\zeta_j)
\end{equation}
and a restricted Stolz cone $S_{A, B}(\zeta)$ for $0 < A < \infty$, $1 \leq B < \infty$ at $\zeta \in \mathbb T^n$ by
\begin{equation}\label{restStolz}
S_{A, B}(\zeta) =  \left\{  (r_1e^{i\theta_1}, \ldots, r_n e^{i\theta_n})  \in S_A(\zeta) :   
\max_{j,k}  \frac{1-r_j}{1-r_k} \leq B \right\}.
\end{equation}

It is an elementary fact that there is a constant $\delta_A > 0$ such that 
\begin{equation}\label{Stolzest1}
\delta_A \vert 1 - re^{i\varphi} \vert \leq \inf_{\vert \theta \vert \leq A (1-r)} \vert 1 - re^{i(\varphi - \theta)} \vert \leq 
 \vert 1 - re^{i\varphi} \vert , \quad  0 < A < \infty.
\end{equation}

Let us set 
\begin{equation}\label{cbft}
 c({\bf t}) = \sum_{j=1}^n (t_j + 2),  \qquad {\bf t} \in (-1, +\infty)^n.
\end{equation}

\begin{proposition}\label{PropStest3}
Let $\eta = ( e^{i\psi_1}, \ldots, e^{i\psi_n} )$ in $\mathbb T^n$, ${\bf r} \in [0, 1)^n$, ${\bf t} \in (-1, +\infty)^n$
and $0 < A < \infty$. Then we have,
\begin{equation}\label{Stolzest3}
\sup_{\xi :\; {\bf r} \cdot \xi \in S_A(\eta)} \vert P_{\alpha,\beta}[d\mu]({\bf r} \cdot \xi)\vert  \leq  
\frac{k_{\alpha, \beta}}{k_{\bf t} \delta_A^{c({\bf t})}}
  \int_{\mathbb T^n} K_{\bf t} ({\bf r} \cdot \eta, \zeta) d\vert \mu\vert (\zeta),
\quad \eta \in \mathbb T^n.
\end{equation}
Here ${\bf t} = (t_1, \ldots, t_n)$ where $t_j = \Re \alpha_j + \Re \beta_j$ for $1 \leq j \leq n$.
\end{proposition}

\begin{proof}
Using \eqref{Stolzest1} we obtain
\begin{align}\label{Stolzest2}
\sup_{z \in S_A(\eta), \vert z_j \vert = r_j} U_{\bf t} (z \cdot \overline\zeta)& = 
\sup_{z \in S_A(\eta), \vert z_j \vert = r_j} \prod_{j=1}^n 
\frac{   (1 - r_j^2)^{t_j + 1}  }{  \vert  1 - z_j \overline{\zeta_j}  \vert^{t_j + 2}   }         \nonumber            \\
& = \sup_{\vert \theta_j \vert \leq A (1 - r_j)}
\prod_{j=1}^n \frac{   (1 - r_j^2)^{t_j + 1}  }{  \vert  1 - r_j  e^{i  [ (\psi_j - \theta_j) -\varphi_j ]}   \vert^{t_j + 2} }
\nonumber \\
& \leq    \delta_A^{-c({\bf t})} 
\prod_{j=1}^n \frac{   (1 - r_j^2)^{t_j + 1}  }{  \vert  1 - r_j  e^{i   (\psi_j  -\varphi_j) }   \vert^{t_j + 2} } \nonumber \\
& = \delta_A^{-c({\bf t})} U_{\bf t} ( {\bf r} \cdot \eta \cdot \overline\zeta), 
\end{align}
where 
$z = (r_1 e^{i ( \psi_1 - \theta_1)}, \ldots, r_n e^{i ( \psi_n - \theta_n)} ) $ is in $ S_A(\eta)$ and 
$\zeta = ( e^{i\varphi_1}, \ldots, e^{i\varphi_n})$ is in $\mathbb T^n$. This and
Lemma \ref{Muz} give desired estimate.
\end{proof}

Note that in \eqref{Stolzest3} both $\eta \in \mathbb T^n$ and ${\bf r} \in [0, 1)^n$ are fixed. In order to estimate the
integral on the right hand side of \eqref{Stolzest3} we need a special partition of $\mathbb T^n$ adapted to given $\bf r$.

\subsection{Partitions $\mathcal P_{\bf r}$ of $\mathbb T^n$ and estimates of $\vert \mu \vert (Q)$, 
$Q \in \mathcal P_{\bf r}$}

For a given $0 \leq \rho < 1$ we denote by $\kappa = \kappa (\rho)$ the unique $\kappa \in [1, 2)$ such that 
\begin{equation}\label{kapparho}
\frac{\pi}{(1-\rho) \kappa (\rho)} = 2^p \qquad \mbox{where} \quad  p \in \mathbb N_0.
\end{equation}
This $p = p(\rho)$ is uniquely determined by $\rho$.  We set
\begin{equation}\label{rhopart}
x_0 = 0, \quad x_j = x_j(\rho) = 2^{j-1} (1-\rho)\kappa (\rho) \quad \mbox{for} \quad
1 \leq j \leq p+1.
\end{equation}
Clearly $0 = x_0 < x_1 < \ldots < x_{p+1} = \pi$ and $2x_j = x_{j+1}$ for $1 \leq j \leq p$. Set,  for $j = 0, \ldots, p$,
$J_j^+ = J_j^+ (\rho)  = [x_j, x_{j+1})$ and $J_j^- = J_j^- (\rho) = [-x_{j+1}, -x_j)$. These 
$2(p+1)$ semi open intervals form a partition $\tilde{\mathcal P}_\rho$ of $[-\pi, \pi)$ and the corresponding $2(p+1)$ semi open arcs $I_j^+ = I_j^+(\rho)$ and $I_j^- = I_j^-(\rho)$ on $\mathbb T$ form a partition $\mathcal P_\rho$ of 
$\mathbb T$. Note that $m_1(I_j^{\pm}) = 2^{j-2} (1-\rho)\kappa (\rho)/ \pi$ for $1 \leq j \leq p$ and
$m_1(I_1^{\pm}) = m_1(I_0^{\pm})$.

Let us fix $1 \leq B < +\infty$ and choose the smallest integer $b \geq 0$ such that $B \leq 2^b$. Assume ${\bf r} \in [0, 1)^n$
satisfies condition $\max_j (1 -  r_j) \leq B \min_j (1 - r_j)$. We can assume, without loss of generality, that $r_1 \leq r_2 \leq
\cdots \leq r_n$. The partitions $\mathcal P_{r_k}$, $1 \leq k \leq n$, of $\mathbb T$ induce a partition of $\mathbb T^n$. Namely, choose $\varepsilon$ in $\{+, - \}^n$, 
$j$ in $\{0, \ldots, p(r_1)\} \times \cdots \times \{0, \ldots, p(r_n)\} $ and set 
$Q(\varepsilon, j) = \prod_{l=1}^n I_{j_l}^{\varepsilon_l}$.
These $[2(p(r_1)+1)] \times \cdots \times [2(p(r_n)+1)]$ boxes form a partition $\mathcal P_{\bf r}$ of $\mathbb T^n$.
Moreover, for each choice of $\varepsilon$ and $j$ there is an increasing sequence $\iota = (\iota_l)_{l=1}^n$ of integers such that $\iota_1 = 0$, $\iota_n \leq b$ and $Q(\varepsilon, j)$ is a $(j_1 + \iota_1, \ldots, j_n+\iota_n)$ box.

 Every box $Q(\varepsilon, j)$ is contained in a box $R(j)$ 
consisting of all $(e^{i\theta_1}, \ldots, e^{i\theta_n})$ such that $-x_{j_k+1}(r_k) \leq \theta_k < x_{j_k+1}(r_k)$ for all 
$1 \leq k \leq n$. The box $R(j)$ is, being homothetic to $Q(\varepsilon, j)$, also a $j + \iota$ box, it is clearly centered at 
$\bf 1$ and $m_n(R(j)) = 4^n m_n(Q(\varepsilon, j))$. Hence, using \eqref{shrink} we obtain 

\begin{align}
\vert \mu \vert (Q(\varepsilon, j)) & \leq \vert \mu \vert (R(j)) \leq m_n(R(j)) (M_{j + \iota} \mu)({\bf 1}) \label{Rj} \\ 
& =  4^n m_n(Q(\varepsilon, j) ) (M_{j + \iota} \mu)({\bf 1})  \nonumber \\
& \leq  4^n 2^{(n-1)b} m_n(Q(\varepsilon, j) ) (M_j \mu)({\bf 1}) \nonumber \\
& \leq 4^n  (2B)^{n-1}  m_n(Q(\varepsilon, j) ) (M_j \mu)({\bf 1}) \label{Rj2}
\end{align}

It easily follows from \eqref{rhopart} and \eqref{kapparho} that
\begin{equation}\label{mnqej}
m_n(Q(\varepsilon, j)) = 4^{-n} 2^{j_1 + \cdots + j_n} \prod_{k=1}^n \frac{(1 - r_k) \kappa (r_k)}{\pi}
= 4^{-n} \prod_{k=1}^n 2^{j_k - p(r_k)} 
\end{equation}
which, combined with the previous inequality, gives the following proposition.
\begin{proposition}\label{Mjmu}
Let $\mu$ be a complex Borel measure on $\mathbb T^n$, let $1 \leq B < \infty$, assume ${\bf r} \in [0, 1)^n$ satisfies condition $\max_j (1-r_j) \leq B \min_j (1-r_j)$ and let $\mathcal P_{\bf r}$ be the partition of $\mathbb T^n$ associated with
$\bf r$. Then for every $\varepsilon \in \{ -1, +1 \}^n$ and $j = (j_1, \ldots, j_n)$ such that $0 \leq j_k \leq p(r_k)$ for all
$1 \leq k \leq n$ we have:
\begin{equation}\label{mjmu}
\vert \mu \vert (Q(\varepsilon, j)) \leq (2B)^{n-1}  2^{j_1 + \cdots + j_n} \prod_{k=1}^n \frac{(1 - r_k) \kappa (r_k)}{\pi}
 (M_j \mu)({\bf 1}) .
\end{equation}
\end{proposition}

\subsection{Estimate of $P_{\alpha, \beta}[d\mu]$ over $S_{A, B}(\eta)$}

Let us fix $0 \leq \rho < 1$ and consider points $x_0, \ldots x_{p+1}$ from \eqref{rhopart} which generate partition 
$\tilde{\mathcal P}_\rho$ of $[0, \pi)$. We apply \eqref{utest} to $\theta = x_1, \ldots, x_p$ and obtain,
using \eqref{rhopart} and $\kappa (\rho) \geq 1$, the following inequality for $1 \leq j \leq p = p(\rho)$:

\begin{equation}\label{atpointsxj}
u_t(\rho e^{i x_j})  \leq \frac{2^{t+1} \pi^{t+2}}{1-\rho} \left( \frac{1 - \rho}{ 2^{j-1} (1-\rho) \kappa (\rho)} 
\right)^{t+2}  \leq \frac{2^{2t+3} \pi^{t+2}}{1-\rho} 2^{-(t+2)j}. 
\nonumber
\end{equation}
It is valid also for $j = 0$ by definition
\eqref{utdim1} of $u_t$, so we obtained
\begin{equation}\label{atpointsxj}
u_t(\rho e^{i x_j})  \leq  \frac{2^{2t+3} \pi^{t+2}}{1-\rho} 2^{-(t+2)j}, \qquad   0 \leq j \leq p = p(\rho).
\end{equation}
Note that $u_t(\rho e^{i\theta})$ is even in $\theta \in (-\pi, \pi)$  and decreasing in $\theta \in (0, \pi)$ for a fixed 
$0 \leq \rho < 1$ (see \cite{OW}). Thus \eqref{atpointsxj} gives, for ${\bf t}$ in $(-1, +\infty)^n$
and ${\bf r}$ in $[0, 1)^n$:
\begin{align}\label{utonqej}
U_{\bf t} ({\bf r} \cdot \zeta) & \leq C_{\bf t}\prod_{k=1}^n 2^{-(t_k + 2)j_k}  \prod_{k=1}^n \frac{1}{1-r_k}, 
\qquad \zeta \in Q(\varepsilon, j).  
\end{align}

Now we assume $\bf r$ satisfies condition $\max_j (1-r_j) \leq B \min_j (1-r_j)$ where $B \geq 1$ and
estimate the integral appearing in Proposition \ref{PropStest3}, 
taking $\eta = {\bf 1}$ for convenience of notation. Using
\eqref{utonqej} and Proposition \ref{Mjmu} we obtain
\begin{align} \label{estintKt}
\int_{\mathbb T^n} K_{\bf t} ({\bf r} \cdot {\bf 1}, \zeta) d\vert \mu\vert (\zeta) & = k_{\bf t} \int_{\mathbb T^n}
U_{\bf t} ({\bf r} \cdot \zeta)  d\vert \mu\vert (\zeta) = k_{\bf t} \sum_{Q \in \mathcal P_{\bf r}} \int_Q
U_{\bf t} ({\bf r} \cdot \zeta)  d\vert \mu\vert (\zeta)  \nonumber \\
& \leq k_{\bf t} C_{\bf t}  \prod_{k=1}^n \frac{1}{1-r_k}   \sum_{(\varepsilon, j)}  \vert \mu \vert (Q(\varepsilon, j))  
\prod_{k=1}^n 2^{-(t_k + 2)j_k}
   \\
& \leq 2^{2n-1} B^{n-1} k_{\bf t} C_{\bf t}  \sum_j \prod_{k=1}^n 2^{-(t_k + 2)j_k}  2^{j_k} 
\frac{\kappa(r_k)}{\pi} (M_j \mu)({\bf 1})   \nonumber \\
& \leq \frac{2^{3n-1} B^{n-1} }{\pi^n}  k_{\bf t} C_{\bf t}  \sum_j   \prod_{k=1}^n  2^{-(t_k + 1)j_k}  (M_j \mu)({\bf 1}) \nonumber \\
& \leq \frac{2^{3n-1} B^{n-1} }{\pi^n}  k_{\bf t} C_{\bf t} (M_q \mu)({\bf 1})
 \nonumber 
\end{align}
where
$$
\quad q = q({\bf t}) = 2^{- \min_{1 \leq k \leq n} (t_k + 1)} < 1.
$$

Note that the right hand side does not depend on ${\bf r}$. Now we use \eqref{Stolzest3} to get
\begin{equation}\label{MNTat1}
\sup_{{\bf r} \cdot \xi \in S_{A, B}({\bf 1})} \vert P_{\alpha,\beta}[d\mu]({\bf r} \cdot \xi)\vert  \leq  
\frac{k_{\alpha, \beta}}{\delta_A^{c({\bf t})}} \frac{2^{3n-1}  B^{n-1}}{\pi^n}  C_{\bf t} (M_q \mu)({\bf 1}).
\end{equation}
Since \eqref{MNTat1} holds at any point $\eta \in \mathbb T^n$ we proved the following proposition.
\begin{proposition}\label{PmuMq}
Let $0 < A < +\infty$ and $1 \leq B < +\infty$. Set $q = 2^{- m(\alpha, \beta)} < 1$ where 
$m(\alpha, \beta) = \min_{1 \leq k \leq n} (\Re \alpha_k + \Re \beta_k + 1)$, 
 $t_j = \Re \alpha_j + \Re \beta_j$ for $1 \leq j \leq n$ and $c({\bf t}) = \sum (t_j + 2)$. 
Then there are constants $\delta_A > 0$ and $C_{\alpha, \beta} < \infty$ such that
\begin{equation}\label{pmumq}
\sup_{z \in S_{A, B} (\eta)} \vert P_{\alpha,\beta}[d\mu](z)\vert  \leq  B^{n-1}
\frac{C_{\alpha, \beta}}{  \delta^{c({\bf t})}_A   } (M_q \mu)(\eta), \qquad \mu \in \mathcal M (\mathbb T^n), 
\quad \eta \in \mathbb T^n.
\end{equation}
\end{proposition}


\subsection{Restricted non-tangential maximal function}

In view of Proposition \ref{PmuMq} it is natural to introduce a restricted non-tangential maximal function of aperture 
$0 < A < +\infty$ and restriction $1 \leq B < \infty$ of a function $u \in C(\mathbb D^n)$
by $(\tilde M_{A, B}^{NT} u) (\eta) =   \sup_{z \in S_{A, B} (\eta)} \vert u(z) \vert  $, $\eta \in \mathbb T^n$ and of a
complex Borel measure $\mu$ on $\mathbb T^n$ by
\begin{equation}\label{NTmax}
(M_{A, B}^{NT} \mu) (\eta) = \sup_{z \in S_{A, B} (\eta)} \vert P_{\alpha,\beta}[d\mu](z)\vert , \qquad \eta \in \mathbb T^n.
\end{equation}
Of course, $M_{A, B}^{NT} \mu =  \tilde M_{A, B}^{NT}( P_{\alpha, \beta}[d\mu])$.
If $f \in L^1(\mathbb T^n)$ we write simply $M_{A, B}^{NT} f$ instead of $M_{A, B}^{NT} f dm_n$.

Combining Proposition \ref{PmuMq} and Lemma \ref{G} we obtain the following weak $(1, 1)$ type estimate for the restricted
non-tangential maximal function $M_{A, B}^{NT}$.
\begin{theorem}\label{Wk11nt}
For every $0 < A < +\infty$ and $1 \leq B < \infty$ there is a constant $C = C_{\alpha, \beta, A, B}$ such that
\begin{equation}\label{wk11nt}
m_n ( \{ \eta \in \mathbb T^n : (M_{A, B}^{NT} \mu) (\eta) > \lambda \}) 
\leq C \frac{ \| \mu \|}{\lambda}, \qquad \lambda > 0.
\end{equation}
\end{theorem}

\subsection{Applications to Fatou type theorems}

Now we turn to Fatou type theorems, Theorem \ref{Wk11nt} will play a key role.

\begin{definition}
We write $NT_{A, B}- \lim_{z \to \eta} u(z) = l$ if for every $\varepsilon > 0$ there is a $\delta > 0$ such that 
$\vert u(z) - l \vert < \varepsilon$ whenever $\vert z - \eta \vert < \delta$ and $z \in S_{A, B}(\eta)$. Here $u$ is a function on 
$\mathbb D^n$, $0 < A < \infty$, $1 \leq B < \infty$ and $\eta$ is in $\mathbb T^n$. 

If  $NT_{A, B}- \lim_{z \to \eta} u(z) = l$ for every $0 < A < \infty$ and $1 \leq B < \infty$ we write
$RNT- \lim_{z \to \eta} u(z) = l$ (restricted non-tangential limit).
\end{definition}

If $u$ is real valued we can analogously consider $NT_{A, B} - \limsup_{z \to \eta} u(z)$.

\begin{theorem}\label{NTforL1}
Let $f \in L^1(\mathbb T^n)$, $u = P_{\alpha, \beta} [f]$. Then
\begin{equation}
RNT - \lim_{z \to \eta} u(z) = f(\eta) \quad \mbox{for almost every} \quad \eta \in \mathbb T^n.
\end{equation}
\end{theorem}

\begin{proof}
We fix $0 < A <\infty$ and $1 \leq B < \infty$.
Let us define for $\delta > 0$, $g \in L^1(\mathbb T^n)$ and $\eta \in \mathbb T^n$
\begin{equation}\label{omegadelta}
\Omega^\delta (g, \eta) = \sup \{ \vert P_{\alpha, \beta}[g] (z) -  P_{\alpha, \beta}[g] (w) \vert : z, w \in S_{A, B} (\eta)
\cap B(\eta, \delta)  \}
\end{equation}
and note the following properties of $\Omega^\delta (g, \eta)$:

$1^o$. $\Omega^\delta (g, \eta)$ is increasing in $\delta > 0$ and 
$\Omega^\delta (g, \eta) \leq 2 (M_{A, B}^{NT} g) (\eta)$.

$2^o$. $\lim_{\delta \to 0} \Omega^\delta (g, \eta) = 0$ if and only if 
$NT_{A, B} - \lim_{z \to \eta} P_{\alpha, \beta}[g](z)$ exists.

$3^o$. $\lim_{\delta \to 0} \Omega^\delta (g, \eta) = \lim_{\delta \to 0} \Omega^\delta (g - \varphi, \eta)$ for every
$\varphi \in C(\mathbb T^n)$.

The first two properties are obvious, the third one follows from Theorem \ref{nepprod}. Set
$E_\varepsilon = \{ \zeta \in \mathbb T^n : \lim_{\delta \to 0} \Omega^\delta (f, \zeta) > \varepsilon \}$ for $\varepsilon > 0$.
By properties $1^o$ and $3^o$ we have, for any $\varphi \in C(\mathbb T^n)$:
\begin{equation*}
E_\varepsilon = \{ \zeta \in \mathbb T^n : \lim_{\delta \to 0} \Omega^\delta (f - \varphi, \zeta) > \varepsilon \} \subset 
 \{ \zeta \in \mathbb T^n : [M_{A, B}^{NT} (f - \varphi)] (\zeta) > \varepsilon/2 \}.
\end{equation*}
Using Theorem \ref{Wk11nt} we deduce
\begin{equation*}
m_n (  \{ \zeta \in \mathbb T^n : [M_{A, B}^{NT} (f - \varphi)] (\zeta) > \varepsilon/2 \} ) \leq
\frac{2C_{\alpha, \beta, A, B}}{\varepsilon} \| f - \varphi \|_{L^1 (\mathbb T^n)}.
\end{equation*}
Since $C (\mathbb T^n)$ is dense in $L^1 (\mathbb T^n)$, the set 
$ \{ \zeta \in \mathbb T^n : [M_{A, B}^{NT} (f - \varphi)] (\zeta) > \varepsilon/2 \}$ is a null set and therefore $E_\varepsilon$ is also a null set. Since $\varepsilon > 0$ is arbitrary it follows that
 $\lim_{\delta \to 0} \Omega^\delta (f, \eta) = 0$ for almost every $\eta \in \mathbb T^n$. Therefore, using property $2^o$,
$NT_{A, B} - \lim_{z \to \eta} P_{\alpha, \beta}[f](z) = F(\eta)$ exists for almost every $\eta \in \mathbb T^n$.
Now we have $F(\eta) = f(\eta)$ by Theorem \ref{LPnorm}, formula \eqref{Lpconv}.
\end{proof} 

A local version of the above theorem for $n$ - harmonic functions can be found in \cite{Z} (Theorem 4.13, Chapter XVII).
It is not possible to replace $RNT$ limits in the above theorem by non-restricted $NT$ limits, i.e. by limits over full Stoltz cones
$S_A(\eta)$, see \cite{JMZ} and \cite{MZ}.

A special case of the following lemma appeared in \cite{Rud2}  (Lemma 3, page 29) where $n$ - harmonic functions were
considered.

\begin{lemma}\label{vanishes}
If a complex Borel measure $ \mu  $ on $ \mathcal M (\mathbb T^n)$ vanishes in an open set $ V\subset \mathbb{T}^n$ and 
$u = P_{\alpha, \beta}[d\mu]$ then 
\begin{equation}\label{vanisheseq}
RNT - \lim_{z \to \eta} u(z) = 0 \quad \mbox{for almost every} \quad \eta \in V.
\end{equation}
\end{lemma}

\begin{proof}
We fix an aperture $0 < A < \infty$ and $1 \leq B < \infty$.
By Lemma \ref{G} $M_q \mu$ is finite almost everywhere on $\mathbb T^n$ and in particular on $V$.
Hence it suffices to prove that $NT_{A, B} - \lim_{z \to \eta} u(z) = 0$ whenever $\eta$ is a point in $V$ such that
$(M_q \mu)(\eta) < +\infty$.  We assume, for convenience of notation, that $\eta = {\bf 1}$. Since ${\bf 1}$ is in $V$, there
is a positive integer $N$ such that $I_N^n \subset V$, where $I_N = \{ e^{i\theta} : \vert \theta \vert \leq \pi 2^{-N} \}$.
Now assume ${\bf r}$ in $[0, 1)^n$ satisfies $\max_j (1-r_j) \leq B \min_j (1-r_j)$.
Using \eqref{Stolzest3}, \eqref{estintKt}  and  \eqref{Rj} we obtain 
\begin{align} \label{forsing}
\sup_{\xi :\; {\bf r} \cdot \xi \in S_{A, B} ({\bf 1})} \vert P_{\alpha,\beta}[d\mu]({\bf r} \cdot \xi) \vert 
& \leq C_{\bf t} \frac{k_{\alpha, \beta}}{\delta_A^{c({\bf t})}      }
\prod_{k=1}^n \frac{1}{1-r_k}   \sum_{(\varepsilon, j)} \vert \mu \vert (Q(\varepsilon, j)) \prod_{k=1}^n  
\frac{1}{2^{(t_k + 2)j_k}} \nonumber  \\
& \leq  C_{\bf t} 2^n  \frac{k_{\alpha, \beta}}{\delta_A^{c({\bf t})}       }
\prod_{k=1}^n \frac{1}{1-r_k}   \sum_{j \in I({\bf r})} \vert \mu \vert  (R(j))  \prod_{k=1}^n  \frac{1  }{2^{(t_k + 2)j_k}},
\end{align}
where $I({\bf r}) = \{0, \ldots, p(r_1)\} \times \cdots \times \{0, \ldots, p(r_n)\}$.
We see, using  \eqref{kapparho} and \eqref{rhopart}, that $\vert \mu \vert (R(j)) = 0$ whenever $j_k < p(r_k) - N$ for all
$1 \leq k \leq n$. Next, for ${\bf r} \in [0, 1)^n$ we set $Z( {\bf r} ) =  \{ (j_1, \ldots, j_n) \in \mathbb Z_+^n : 
j_k < p(r_k) - N,     k = 1, \ldots, n \} $. Therefore, arguing as in derivation of \eqref{MNTat1}, 
we deduce from \eqref{forsing}:
\begin{align}\label{remainder}
\varphi ({\bf r}) & = \sup_{\xi :\; {\bf r} \cdot \xi \in S_{A, B} ({\bf 1})} \vert P_{\alpha,\beta}[d\mu]({\bf r} \cdot \xi) \vert
\nonumber \\
& \leq C(\alpha, \beta, A)  \prod_{k=1}^n \frac{1}{1-r_k}   
\sum_{j \in \mathbb Z_+^n \setminus Z( {\bf r} )} \vert \mu \vert  (R(j))  \prod_{k=1}^n  \frac{1}{2^{(t_k + 2)j_k}}
\nonumber \\
&  \leq C(\alpha, \beta, A) B^{n-1} \sum_{j \in \mathbb Z_+^n \setminus Z( {\bf r} )} q^{\vert j \vert}  (M_j \mu) ({\bf 1}).
\end{align}
By our assumption $(M_q \mu) ({\bf 1}) = \sum_{j \in Z_+^n} q^{\vert j \vert} (M_j \mu) ({\bf 1}) < +\infty$. Since
$p(\rho) $ tends to $ +\infty$ as $\rho \uparrow 1$ we see that for every
finite $F \subset \mathbb Z_+^n$ there is an $\delta > 0$ such that $F \subset Z ( {\bf r} )$ whenever $1-r_k < \delta$ for all
$1 \leq k \leq n$. Thus the sum appearing in \eqref{remainder} can be seen as a remainder of a convergent sum, it tends to
zero as ${\bf r} \to {\bf 1}$. Hence $\lim_{\bf r \to 1} \varphi ({\bf r}) = 0$ which is equivalent to 
$NT_{A, B} -\lim_{z \to {\bf 1}} u(z) = 0$.
\end{proof}

Convergence to $0$ at every point $\eta$ in $V$
fails even in the special case of $n$ - harmonic functions (that is, $\alpha = \beta = 0$).
See \cite{Rud2}, page 25, for an example. This failure of localization is specifically multidimensional phenomenon: if $n=1$ there are pointwise results of this nature, for example Theorem 6.2 from \cite{OW}.

\begin{theorem}\label{NTsing}
Assime  $ \mu \in \mathcal M (\mathbb T^n) $ is singular with respect to $ m_n $. Then 
\begin{equation*}
RNT- \lim_{z \to \zeta}  P_{\alpha,\beta}[d\mu](z)=0  \quad \mbox{for almost every}\quad \zeta \in  \mathbb T^n.
\end{equation*}
\end{theorem}

\begin{proof}
We fix $0 < A < \infty$ and $1 \leq B < \infty$ and set for $\nu \in \mathcal M(\mathbb T^n)$ and $\delta > 0$
\begin{equation}\label{lambdaset}
\Lambda_{A, B}  (\nu, \delta) = \{ \zeta \in \mathbb T^n :  
NT_{A, B}  - \limsup_{z \to \zeta} \vert  P_{\alpha,\beta}[d\nu](z)  \vert > \delta  \},
\end{equation}
\begin{equation}\label{Phiset}
\Phi_{A, B}  (\nu, \delta) = \{ \zeta \in \mathbb T^n :  (M_{A, B}^{NT} \nu)  (\zeta) > \delta  \}.
\end{equation}
Clearly $\Lambda_{A, B} (\nu, \delta) \subset \Psi_{A, B} (\nu, \delta)$.
It suffices to prove that $\Lambda_{A, B}  (\mu, \delta)$ is a null set for every $\delta >  0$. Let us fix $\delta > 0$. By Lemma
\ref{vanishes} the sets $\Lambda_{A, B} (\nu, \delta)$ and $\Lambda_{A, B} (\nu - \sigma, \delta)$ are equal up to a set 
of measure zero whenever $\sigma \in \mathcal M(\mathbb T^n)$ is supported on a compact $K$ of Lebesgue measure $0$. Hence, for such $\sigma$ we have, using Theorem \ref{Wk11nt}:
\begin{align}
m_n( \Lambda_{A, B} (\mu, \delta))  & = m_n( \Lambda_{A, B} (\mu - \sigma, \delta)) 
\leq m_n( \Psi_{A, B} (\mu - \sigma, \delta)) \nonumber \\
& \leq C_{\alpha, \beta, A, B} \frac{ \| \mu - \sigma \|}{\delta}. \label{NTsingeq}
\end{align}
By regularity properties of complex Borel measures, for a given $\varepsilon > 0$ there is a compact $K \subset \mathbb T^n$
such that $m_n (K) = 0$ and $\| \mu - \mu_K \| < \varepsilon$, where $\mu_K (E) = \mu (E \cap K)$. Hence, taking
$\sigma = \mu_K$ in \eqref{NTsingeq} we see that $\Lambda_{A, B} (\mu, \delta)$ is contained in a set of arbitrarily small measure, hence  $\Lambda_{A, B} (\mu, \delta)$ is a null set.
\end{proof}

Combining Theorem \ref{NTforL1} and Theorem \ref{NTsing} we obtain the following corollary. In the special case of $n$ - 
harmonic functions it appeared in \cite{Z1}. However, for $(\alpha, \beta)$ - harmonic functions it is new even in dimension 1, although in   \cite{OW}       one parameter case was studied.

\begin{corollary}\label{Combined}
Let $\mu \in \mathcal M (\mathbb T^n)$ and $u = P_{\alpha, \beta}[d\mu]$. Then 
\begin{equation}\label{combinedeq}
RNT - \lim_{z \to \zeta} u(z) = f(\zeta) \qquad \mbox{for almost every} \quad \zeta \in \mathbb T^n
\end{equation}
where $d\mu = fdm_n + d\sigma$ is Lebesgue - Radon - Nikodym decomposition of $d\mu$ into absolutely continuous part
$fdm_n$ and singular part $d\sigma$.
\end{corollary}


\begin{proposition}
Let $u \in sh_{\alpha, \beta} (\mathbb D^n)$ and assume $M^+ u (\zeta) = \sup_{0 \leq r < 1} \vert u(r\zeta) \vert$ is
in $L^1(\mathbb T^n)$. Then $u = P_{\alpha, \beta} [f]$ for some $f \in L^1 (\mathbb T^n)$.
\end{proposition}

\begin{proof}
Clearly  $u \in sh^1_{\alpha, \beta} (\mathbb D^n)$ and therefore $u = P_{\alpha, \beta} [d\mu]$ for some $\mu$ in
$\mathcal M (\mathbb T^n)$. We have Lebesgue - Radon - Nikodym decomposition $d \mu = f dm_n + d \sigma$ and by
Corollary \ref{Combined} we have $\lim_{r \to 1} u(r\zeta) = f(\zeta)$ for almost every $\zeta \in \mathbb T^n$. Since
$M^+ u$ is integrable, the Dominated Convergence Theorem imples that convergence is also in $L^1 (\mathbb T^n)$ norm.
Since $u_r dm_n \to d\mu$ in the weak$^\ast$ sense, this implies $d\mu = f dm_n$.
\end{proof}

\subsection*{Acknowledgements}
The second and the third author are supported by MNTPR, Serbia








\begin{thebibliography}{}




\bibitem{ABC}
P. AHERN, J. BRUNA and C. CASCANTE, \emph{ $ H^p$-theory for generalized M-harmonic functions in the unit ball}, Indiana Univ. Math. J. 45 (1996), 103--135.

\bibitem{B}
H. BATEMAN, \emph{ Higher Transcendental Functions,} McGraw-Hill Book Company, New York, 1953.

\bibitem{CZ} A. P. CALDER$\acute{O}$N and A. ZYGMUND \emph{ Note on the boundary values of functions of several complex variables,} Contributions to Fourier Analysis (AM-25), Princeton University Press, Princeton 1950, pp. 145--165.
https://doi.org/10.1515/9781400881956-005
\bibitem{Da}  C. S. DAVIS, \emph{Iterated Limits in $N^\ast (U^n)$},
Trans.  Am. Math. Soc. 178 (1973), 139-146.


\bibitem{Ge}
D. GELLER, \emph{ Some results in $ H_p$ theory for the Heisenberg group,} Duke Math. J. 47(1980), 365--390. https://doi.org/10.1215/S0012-7094-80-04722-5

\bibitem{GuSt}
R. F. GUNDY and E.M. STEIN, \emph{ $H^p$ theory for the poly-disc (biharmonic functions/area integral/Brownian motion)},
Proc. Natl. Acad. Sci. 76(1979), no.~ 3, 1026--1029. https://doi.org/10.1073/pnas.76.3.1026


\bibitem{MA}
A. KHALFALLAH, M. MATELJEVI\'C and M MHAMDI, \emph{ Some properties of mappings admitting general Poisson representations,} Mediterr. J. Math. (2021). https://doi.org/10.1007/s00009-021-01827-0

\bibitem{JMZ} B. JESSEN, J. MARCINKIEWICZ and A. ZYGMUND, \emph{ Note on the Differentiability of Multiple Integrals,} Fundamenta Mathematicae, 25 (1935), 217-- 234.


\bibitem{KO} 
M. KLINTBORG and A. OLOFSSON,\emph{ A series expansion for generalized harmonic functions},
Analysis and Mathematical Physics  (2021). 
https://doi.org/10.1007/s13324-021-00561-w

\bibitem{MZ}
J. MARCINKIEWICZ and A. ZYGMUND\emph{ On the summability of double Fourier series,} Fundamenta Mathematicae 32(1939), 112 --132.



\bibitem{OP} A. OLOFSSON, \emph{ Differential operators for a scale of Poisson type kernels in the
unit disc,} J. Anal. Math. 123(2014), 227–-249. https://doi.org/10.1007/s11854-014-0019-4

\bibitem{OW} A. OLOFSSON and J. WITTSTEN, \emph{ Poisson integrals for standard weighted Laplacians in the unit disc,} J. Math.
Soc. Jpn. 65(2013), no.~2, 447–-486.  https://doi.org/10.2969/jmsj/06520447

 \bibitem{Rud2} 
  W. RUDIN,\emph{ Function Theory in Polydiscs,} W.A. Benjamin  New-York, 1969.

\bibitem{Seur}
K. R. SHRESTHA, \emph{ Hardy Spaces on the Polydisk,} European Journal of Pure and Applied Mathematics 9(2016), no.~3, 292--304.
\bibitem{St1} E. M. STEIN, \emph{ Boundary Behavior of Harmonic Functions on Symmetric Spaces: Maximal Estimates for Poisson Integrals}, Invent. math. 74(1983), 63-83. https://doi.org/10.1007/BF01388531
\bibitem{Z1} A. ZYGMUND \emph{ On the Summability of Multiple Fourier Series,} American
Journal of Mathematics, 69 (1947),  836--850. 
https://doi.org/10.2307/2371803

\bibitem{Z} A.  ZYGMUND, \emph{ Trigonometric Series: Vol.  II,} Cambridge University Press, 1959.

\end{thebibliography}
\end{document}